\documentclass[11pt]{amsart}
\usepackage[utf8]{inputenc}
\usepackage{tikz}
\usepackage[backend=biber,style=alphabetic,sorting=nyt,maxbibnames=99]{biblatex}  
\usepackage{multirow}
\usepackage{textcomp}
\usepackage{csquotes} 
\usepackage{listings}
\usepackage{hyperref}
\hypersetup{colorlinks=false}
\usepackage{amssymb}
\usepackage{enumitem}
\usepackage{graphicx}

\DeclareUnicodeCharacter{00FB}{\^{u}}
\DeclareUnicodeCharacter{0229}{\c{e}}
\DeclareUnicodeCharacter{0159}{\v{r}}
\DeclareUnicodeCharacter{0142}{\l}
\DeclareUnicodeCharacter{015A}{\'{S}}
\DeclareUnicodeCharacter{015B}{\'{s}}
\DeclareUnicodeCharacter{011B}{\v{e}}
\DeclareUnicodeCharacter{00FD}{\'{y}}

\newtheorem{theorem}{Theorem}[section]
\newtheorem{thmx}{Theorem}

\newtheorem{proposition}[theorem]{Proposition}
\newtheorem{lemma}[theorem]{Lemma}
\newtheorem{claim}{Claim}

\theoremstyle{definition}
\newtheorem{definition}[theorem]{Definition}

\newtheorem{remark}[theorem]{Remark}
\newtheorem{fact}{Fact}


\title{A complete metric space without non-trivial separable Lipschitz retracts.}
\author{Petr H\'ajek}\thanks{This research was supported by CAAS CZ.02.1.01/0.0/0.0/16-019/0000778 and by the project SGS21/056/OHK3/1T/13.}
\address[P. H\'ajek]{Czech Technical University in Prague, Faculty of Electrical Engineering.
Department of Mathematics, Technick\'a 2, 166 27 Praha 6 (Czech Republic)}
\email{hajek@math.cas.cz}

\author{Andr\'es Quilis}\thanks{The second author's research has been supported by PAID-01-19}
\address[A. Quilis]{Universitat Polit\`ecnica de Val\`encia. Instituto Universitario de Matem\'atica Pura y Aplicada, Camino de Vera, s/n
46022 Valencia (Spain); and Czech Technical University in Prague, Faulty of Electrical Engineering. Department of Mathematics, Technick\'a 2, 166 27 Praha 6 (Czech Republic)}
\email{anquisan@posgrado.upv.es}

\begin{document}
\begin{abstract}
    We construct a complete metric space $M$ of cardinality continuum such that every non-singleton 
 closed separable subset of $M$ fails to be a Lipschitz retract of $M$. This provides a metric analogue to
 the various classical and recent examples of Banach spaces failing to have linearly complemented subspaces of prescribed smaller density character.
 \end{abstract}
\maketitle

\section{Introduction}
Given two metric spaces $M$ and $N$, a map $F\colon M\rightarrow N$ is said to be \emph{Lipschitz} if there exists a constant $C>0$ such that $d\big(F(x),F(y)\big)\leq C d(x,y)$ for all $x,y\in M$. The \emph{Lipschitz constant} of $F$, denoted $\|F\|_\text{Lip}$, is the smallest number verifying this inequality, i.e.:
$$ \|F\|_\text{Lip}=\sup \bigg\{\frac{d\big(F(x),F(y)\big)}{d(x,y)}\colon x,y\in M,~ x\neq y\bigg\}.$$
We say that a map $F$ is $K$-Lipschitz for $K>0$ if $\|F\|_\text{Lip}\leq K$.

 Given a metric space $M$ and its closed subset $S$, we say that  a Lipschitz map $R\colon M\rightarrow S$ 
is a (Lipschitz) \emph{retraction} from $M$ onto $S$ if
 $R(x)=x$ for all $x\in S$. If there exists a $K$-Lipschitz retraction $R\colon M\rightarrow S$ for some $K\geq 1$, then we say that $S$ is a \emph{$K$-Lipschitz} retract of $M$. Every singleton is trivially a Lipschitz retract in every metric space. 

A search for nontrivial retracts is very natural, as they provide the grip on the structure of the original metric space $M$. In the linear setting (when the metric spaces
are Banach spaces and retractions are bounded projections) the study of projections (i.e complemented subspaces) is one of the main themes of the theory.
The main result of this paper can  be stated as follows.

\begin{thmx}
\label{maintheorem}
There exists a complete metric space $M$ of cardinality continuum such that every non-singleton 
closed separable subset of $M$ fails to be a Lipschitz retract of $M$.
\end{thmx}

The basic ingredient for our construction is provided by a modified  example from Theorem 3.7 in \cite{HajQui22}. In the first two sections  we generalize some of the arguments
from \cite{HajQui22} and then we pass to the transfinite construction of the final example $M$. Our construction is self-contained but quite technical.

Let us put this example into context of the mentioned (linear or non-linear) Banach space situation.
In nonlinear Banach space theory, the study of the Lipschitz structure of Banach spaces is a classical topic with many deep results and open questions (we refer to \cite{BenLin00} for a comprehensive exposition of nonlinear Banach space theory). One of such important problems, going back at least to the seminal paper \cite{Lin64} by J. Lindenstrauss is whether every Banach space is a Lipschitz retraction of its bidual. In \cite{Kal11}, N. Kalton proved that this fails for nonseparable spaces, while the separable case remains open. Let us briefly discuss a consequence of this conjecture for separable Banach spaces, which illustrates one of the main motivations the result in this paper:

Given a Banach space $X$ and $\lambda\geq 0$, a subspace $Y$ of $X$ is said to be \emph{$\lambda$-locally complemented} if $Y^{**}$ is linearly $\lambda$-complemented in $X^{**}$. It follows from the J. Lindenstrauss and L. Tzafriri characterization of Hilbert spaces by closed subspaces with the Compact Extension Property (CEP) (in \cite{LinTza71}) and the equivalence between local complementability and the CEP (due to Kalton in \cite{Kal84}) that in a non-Hilbert Banach space there always exist separable closed subspaces that are not locally complemented. 

It can be shown that every Lipschitz retract of a Banach space is locally complemented (see for instance \cite{HajQui22}). Conversely, if a subspace $Y$ is locally complemented in a Banach space $X$, then $Y$ is a Lipschitz retract of $X$ whenever $Y$ is a Lipschitz retract of its bidual $Y^{**}$. This follows directly considering the restriction of the map $P\circ R_Y\colon X^{**}\rightarrow Y$ to $X$, where $P\colon X^{**}\rightarrow Y^{**}$ is a linear projection onto $Y^{**}$ and $R_Y\colon Y^{**}\rightarrow Y$ is a Lipschitz retraction onto $Y$. 

This is an interesting consequence because every Banach space has a relatively rich structure of locally complemented subspaces of any density character. Indeed, it is a classic result of S. Heinrich and P. Mankiewicz (\cite{HeiMan82}) that in every Banach space $X$, given a closed subspace $Z$ one can always find a closed subspace $Y$ containing $Z$ and with $\text{dens}(Y)=\text{dens}(Z)$ such that $Y$ is $1$-locally complemented in $X$. 

On the other hand, the situation is quite different for the linear structure of Banach spaces. It is well-known
(see \cite{Lin67}) that all infinite-dimensional complemented subspaces of the classical non-separable Banach space $\ell_\infty$ are again isomorphic to $\ell_\infty$, in particular they are non-separable.
A remarkable recent result of P. Koszmider, S. Shelah and M. \'{S}wi\c{e}tek in \cite{KosSheSwi18} shows that, assuming the Generalized Continuum Hypothesis, for every cardinal $\kappa$ there exists a compact topological space $K$ such that the Banach space $C(K)$ does not have any non-trivial complemented subspaces of density character less than or equal to $\kappa$. 

Hence, at least in the separable case, the Lipschitz retractional structure of a Banach space could range from being as rich as the locally complemented structure (if the long standing conjecture of separable Banach spaces being Lipschitz retracts of their bidual holds), to being more similar to the linearly complemented case.

The main result of this paper shows that, in the more general setting of metric spaces where Lipschitz functions are the natural morphisms, there exist metric spaces with no non-trivial separable Lipschitz retracts. The density character of the metric space we construct is the continuum.  

To finish the discussion about the context of this result, we will touch on how it relates to the current research on Lipschitz-free Banach spaces. Given a complete metric space $M$ with a fixed distinguished point $0\in M$, the vector space $\text{Lip}_0(M)=\{f\colon M\rightarrow \mathbb{R}\colon f\text{ is Lipschitz and }f(0)=0\}$ with the norm given by the Lipschitz constant is a Banach space. This Banach space is a dual space, and the \emph{Lipschitz-free} space associated to $M$ is the canonical predual $\mathcal{F}(M)=\overline{\text{span}}\{\delta(x)\colon x\in M\}\subset \text{Lip}_0(M)^*$ of $\text{Lip}_0(M)$, where $\delta(x)\colon \text{Lip}_0(M)\rightarrow \mathbb{R}$ is the Dirac measure at the point $x\in M$ verifying $\langle f,\delta(x)\rangle=f(x)$ for all $f\in \text{Lip}_0(M)$. This Banach space is also known as the Arens-Eells space (defined  in \cite{AreEel56}). It  has been studied extensively in the last decades, especially after the publication of \cite{GodKal03} by G. Godefroy and N. Kalton in 2003.

The main property of Lipschitz-free spaces is the fact that given two metric spaces $M,N$ and a Lipschitz map $F\colon M\rightarrow N$ there exists a linear operator $\widehat{F}\colon \mathcal{F}(M)\rightarrow\mathcal{F}(N)$ with norm $\|\widehat{F}\|=\|F\|_\text{Lip}$ such that $\widehat{F}\circ \delta_M=\delta_N\circ F$. If the space $N$ is a subset of $M$, then the identity map restricted to $N$ yields an isometric embedding of $\mathcal{F}(N)$ into $\mathcal{F}(M)$, and if $R\colon M\rightarrow N$ is a Lipschitz retraction, the associated map $\widehat{R}\colon \mathcal{F}(M)\rightarrow\mathcal{F}(N)$ is a linear and bounded projection from $\mathcal{F}(M)$ onto $\mathcal{F}(N)$ (when the latter is considered as a subspace of the former). This fact implies that the Lipschitz retractional structure of a metric space passes on to the linear structure of the associated Lipschitz-free space. 

The linear structure of Lipschitz-free spaces has been an active topic of research in the past two decades, starting with \cite{GodKal03}, where it is proven that every separable Banach space can be seen as a $1$-complemented linear subspace of its Lipschitz-free space. Other recent results include for instance the fact that every Lipschitz-free space of a metric space $M$ contains a complemented copy of $\ell_1(\Gamma)$, where $\Gamma$ is the density character of $M$ (see \cite{CutDouWoj16} and \cite{HajNov17}), and that there exists a universal constant $K\geq 1$ such that $\mathcal{F}(C)$ is $K \sqrt{N}$-linearly complemented in $\mathcal{F}(\mathbb{R}^N)$ for every closed subset $C$ of the $N$-dimensional euclidean space $\mathbb{R}^N$ (\cite{LanPer13}). A very small sample of articles that deal with the structure of Lipschitz-free spaces include \cite{AliNouPetPro21,Dal15,DalKauPro16,DutFer06,GodOza14,Kal04,Kal11,Kau15}. 

A Banach space $X$ has the \emph{Separable Complementation Property} (SCP for short) if every separable subspace $Z$ is contained in a separable subspace which is complemented in $X$. A Banach space is said to be \emph{Plichko} if there exists a pair $(\Delta,N)$ where $\Delta$ is a linearly dense subset of $X$ and $N$ is a norming subspace of $X^*$ such that the set $\{x\in \Delta\colon \langle f,x\rangle\neq 0\}$ is countable for all $f\in N$. All Plichko Banach spaces have the SCP, and we do not know of any examples of Lipschitz-free spaces (over a metric space) which fail to be Plichko. 

To obtain such an example, a necessary (but not sufficient) condition on the underlying metric space is of course that it must fail to have a nontrivial separable Lipschitz retractional structure. 
We do not know if the Lipschitz-free space associated to the metric space constructed in this article is Plichko (or even if it has the SCP).

The space $\mathcal{F}(\ell_\infty)$ could be a natural candidate for a Lipschitz-free space failing to have the SCP and thus failing to be Plichko, and it seems to be unknown at the moment whether $\mathcal{F}(\ell_\infty)$ does indeed fail these properties. However, P. Kaufmann and L. Candido have recently shown in \cite{CanKau21} that its dual space $\text{Lip}_0(\ell_\infty)$ is linearly isomorphic to $\text{Lip}_0(c_0(\textbf{c}))$, where $\textbf{c}$ denotes the cardinality of the continuum. Since $c_0(\textbf{c})$ is Plichko, we have by Corollary 2.9 in \cite{HajQui22} that $\mathcal{F}(c_0(\textbf{c}))$ is Plichko as well. A similar remark to this effect was already made in \cite{CanKau21}.

Finally, let us discuss the structure of this article. As mentioned, the construction of our metric space is self contained though technical. It is divided into three sections, going from section 2 to section 4. In section 2 we define the basic pieces of the construction, called \emph{threads}. These threads are isometric to subsets of one-dimensional circles with the distance given by the arc-length. We will define uncountable families of totally disconnected threads which verify certain metric properties related to Lipschitz functions between these threads. 

In section 3 we will use these uncountable families of threads to define the building blocks of the final metric space. These building blocks are called \emph{threading spaces}, and each block is built from one of the uncountable families defined in the previous section. All threads that form each one of these threading spaces are attached to two anchor points $\{0,1\}$ in the threading space, and every one of these threading spaces verifies the weaker property that every separable subset containing both anchor points is not a Lipschitz retraction of the whole threading space. 

In section 4 we finish by using these threading spaces to construct the final metric space via a transfinite inductive process of length $\omega_1$. We call the resulting complete metric space the \emph{skein space}. Very informally, the skein space verifies that any pair of points behaves as the pair of anchor points of one of the threading spaces constructed in section 3. This way we have that any separable space with more than one point contains two anchor points of a threading space, and hence it is  not a Lipschitz retract of the whole skein. Although the inductive construction of the skein space is relatively straightforward, proving that it verifies the thesis of Theorem \ref{maintheorem} is quite a technical process, and requires introducing some concepts and using a wide array of techniques, all of which are introduced when needed.

\section{Construction of the fundamental pieces: Threads with infinitely many gaps}
\subsection{Threads}
Let $l,a>0$ with $a\leq l$. We say that a metric space $(T,d_{l,a})$ is an \emph{$\mathbb{R}$-thread of length $l$ and width $a$} if $T$ is a closed subset of the real segment $[0,l]$ containing $0$ and $l$, and the metric $d_{l,a}$ is defined by 
$$ d_{l,a}(x,y)=\min\{|x-y|,x+(l-y)+a,y+(l-x)+a\}$$
for every $x,y\in T$. 
Our main example will be constructed inductively by repeated adjoining of metric spaces, isometric to a thread described above, to the previous space. In this sense, the adjoined new pieces are certainly
meant to be distinct sets. However, keeping in mind this feature, there is no danger of confusion if we 
simply call any metric space $T$ a \emph{thread of length $l$ and width $a$} if $T$ is isometric to an $\mathbb{R}$-thread of length $l$ and width $a$ as defined above (and work with it using the above description). 

Let us mention some basic facts about threads. First, notice that every thread is a compact metric space. Also, we may define in every thread $(T,d_{l,a})$ the natural order and the Lebesgue measure since the set $T$ is a subset of the real line. Then, for every $x,y\in T$ with $x\leq y$ we  define the set $[x,y]_T\subset T$ as $[x,y]\cap T$, where $[x,y]$ is the usual real segment. The set $[x,y]_T$  with the inherited metric is again a thread. 

If $T$ is a thread of length $l$, we say that a closed subset $I$ of $T$ is an \emph{extended interval} of $T$ if $I$ is of the form $[p,q]_T=[p,q]\cap T$ or $[0,p]_T\cup[q,l]_T$ for a pair of points $p,q\in T$ with $p<q$. In either case, the points $p$ and $q$ are called the \emph{extreme points} of $I$. See Figure \ref{threadrep} for a representation of a thread and the two kinds of extended intervals it contains.  

Notice as well that in a thread of length $l$ and width $a$, the distance between the extreme points $0$ and $l$ is exactly the width $a$. We can also realize that every thread is locally isometric to $T$ with the usual metric inherited from the real line; indeed, if the distance between two points of $T$ is less than the width of the thread, then this distance coincides with the usual metric. As a consequence, we have that if the length and the width of a thread coincide, then the thread is isometric to a subset of the real segment $[0,l]$.

\tikzset{every picture/.style={line width=0.75pt}} 
\begin{figure}
    \centering
    
\begin{tikzpicture}[x=0.75pt,y=0.75pt,yscale=-1,xscale=1]

\draw  [draw opacity=0] (230.92,260.41) .. controls (191.54,201.74) and (206.17,122.1) .. (264.27,81.42) .. controls (323.08,40.24) and (404.16,54.57) .. (445.38,113.43) .. controls (477.16,158.81) and (475.95,217.42) .. (446.76,260.76) -- (338.89,187.99) -- cycle ; \draw   (230.92,260.41) .. controls (191.54,201.74) and (206.17,122.1) .. (264.27,81.42) .. controls (323.08,40.24) and (404.16,54.57) .. (445.38,113.43) .. controls (477.16,158.81) and (475.95,217.42) .. (446.76,260.76) ;
\draw   [fill={rgb, 255:red, 74; green, 50; blue, 226 }  ,fill opacity=1 ] (227.96,260.41) .. controls (227.96,258.77) and (229.29,257.44) .. (230.92,257.44) .. controls (232.56,257.44) and (233.88,258.77) .. (233.88,260.41) .. controls (233.88,262.04) and (232.56,263.37) .. (230.92,263.37) .. controls (229.29,263.37) and (227.96,262.04) .. (227.96,260.41) -- cycle ;
\draw   [fill={rgb, 255:red, 74; green, 50; blue, 226 }  ,fill opacity=1 ] (443.8,260.76) .. controls (443.8,259.12) and (445.13,257.8) .. (446.76,257.8) .. controls (448.4,257.8) and (449.73,259.12) .. (449.73,260.76) .. controls (449.73,262.39) and (448.4,263.72) .. (446.76,263.72) .. controls (445.13,263.72) and (443.8,262.39) .. (443.8,260.76) -- cycle ;
\draw  [fill={rgb, 255:red, 208; green, 2; blue, 27 }  ,fill opacity=1 ] (409.29,80.71) .. controls (409.29,79.07) and (410.62,77.74) .. (412.26,77.74) .. controls (413.89,77.74) and (415.22,79.07) .. (415.22,80.71) .. controls (415.22,82.34) and (413.89,83.67) .. (412.26,83.67) .. controls (410.62,83.67) and (409.29,82.34) .. (409.29,80.71) -- cycle ;
\draw  [fill={rgb, 255:red, 208; green, 2; blue, 27 }  ,fill opacity=1 ] (260.08,82.37) .. controls (260.08,80.74) and (261.4,79.41) .. (263.04,79.41) .. controls (264.67,79.41) and (266,80.74) .. (266,82.37) .. controls (266,84.01) and (264.67,85.33) .. (263.04,85.33) .. controls (261.4,85.33) and (260.08,84.01) .. (260.08,82.37) -- cycle ;
\draw  [fill={rgb, 255:red, 74; green, 50; blue, 226 }  ,fill opacity=1 ] (218.63,244.41) .. controls (218.63,242.77) and (219.95,241.44) .. (221.59,241.44) .. controls (223.22,241.44) and (224.55,242.77) .. (224.55,244.41) .. controls (224.55,246.04) and (223.22,247.37) .. (221.59,247.37) .. controls (219.95,247.37) and (218.63,246.04) .. (218.63,244.41) -- cycle ;
\draw  [fill={rgb, 255:red, 74; green, 50; blue, 226 }  ,fill opacity=1 ] (463.96,211.74) .. controls (463.96,210.1) and (465.29,208.78) .. (466.92,208.78) .. controls (468.56,208.78) and (469.88,210.1) .. (469.88,211.74) .. controls (469.88,213.38) and (468.56,214.7) .. (466.92,214.7) .. controls (465.29,214.7) and (463.96,213.38) .. (463.96,211.74) -- cycle ;
\draw  [draw opacity=0] (262.69,82.54) .. controls (263.21,82.16) and (263.74,81.79) .. (264.27,81.42) .. controls (309.64,49.65) and (368.26,50.92) .. (411.61,80.18) -- (338.89,187.99) -- cycle ; \draw  [color={rgb, 255:red, 208; green, 2; blue, 27 }  ,draw opacity=1 ] (262.69,82.54) .. controls (263.21,82.16) and (263.74,81.79) .. (264.27,81.42) .. controls (309.64,49.65) and (368.26,50.92) .. (411.61,80.18) ;
\draw  [draw opacity=0] (466.9,210.91) .. controls (463.75,228.58) and (456.97,245.6) .. (446.76,260.76) -- (338.89,187.99) -- cycle ; \draw  [color={rgb, 255:red, 74; green, 50; blue, 226 }  ,draw opacity=1 ] (466.9,210.91) .. controls (463.75,228.58) and (456.97,245.6) .. (446.76,260.76) ;
\draw  [draw opacity=0] (230.92,260.41) .. controls (227.57,255.41) and (224.6,250.26) .. (222.03,244.99) -- (338.89,187.99) -- cycle ; \draw  [color={rgb, 255:red, 74; green, 50; blue, 226 }  ,draw opacity=1 ] (230.92,260.41) .. controls (227.57,255.41) and (224.6,250.26) .. (222.03,244.99) ;
\draw  [dash pattern={on 4.5pt off 4.5pt}]  (253,260) -- (427,260) ;
\draw [shift={(430,260)}, rotate = 180] [fill={rgb, 255:red, 0; green, 0; blue, 0 }  ][line width=0.08]  [draw opacity=0] (5.36,-2.57) -- (0,0) -- (5.36,2.57) -- cycle    ;
\draw [shift={(250,260)}, rotate = 0] [fill={rgb, 255:red, 0; green, 0; blue, 0 }  ][line width=0.08]  [draw opacity=0] (5.36,-2.57) -- (0,0) -- (5.36,2.57) -- cycle    ;

\draw (245.46,64.38) node [anchor=north west][inner sep=0.75pt]    {$p$};
\draw (419.46,61.71) node [anchor=north west][inner sep=0.75pt]    {$q$};
\draw (200.79,233) node [anchor=north west][inner sep=0.75pt]    {$x$};
\draw (476.12,201) node [anchor=north west][inner sep=0.75pt]    {$y$};
\draw (216.39,262.78) node [anchor=north west][inner sep=0.75pt]    {$0$};
\draw (456.39,260.38) node [anchor=north west][inner sep=0.75pt]    {$l$};
\draw (338,233.4) node [anchor=north west][inner sep=0.75pt]    {$a$};

\end{tikzpicture}
\caption{Thread of length $l$ and width $a$. The red and blue lines correspond to the two kinds of extended intervals possible in a thread.}
\label{threadrep}
\end{figure}
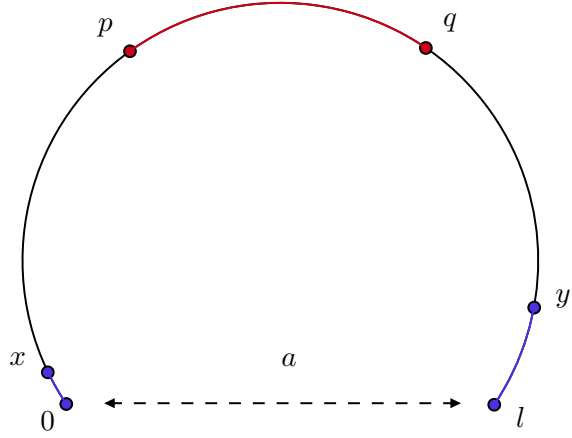

The way we compute the distance in threads implies that Lipschitz functions from threads into other metric spaces are similar to Lipschitz functions from intervals. Specifically we have the following result:
\begin{proposition}
\label{domainthreadisinterval}
Let $T$ be a thread of length $l_T$ and width $a_T$, let $M$ be a metric space, and let $K\geq 0$. A function $F\colon T\rightarrow M$ is $K$-Lipschitz if and only if $d(F(0),F(l_T))\leq K a_T$, and for every $x,y\in T$ we have $d\big(F(x),F(y)\big)\leq K|y-x|$.
\end{proposition}
\begin{proof}
Evidently, if $F$ is $K$-Lipschitz, we directly obtain that $d(F(0),F(l_T))\leq Ka_T$ and the inequality:
$$d\big(F(x),F(y)\big)\leq Kd(x,y)\leq K|y-x|. $$

Suppose now that the inequality is true for every pair of points in $T$, and take $x\leq y\in T$. If $d(x,y)=y-x$, then we obtain directly that $d\big(F(x),F(y)\big)\leq Kd(x,y)$. Otherwise, we have that $d(x,y)=x+a_T+(l_T-y)$. Therefore
\begin{align*}
    d\big(F(x),F(y)\big)&\leq d\big(F(x),F(0)\big)+d\big(F(0),F(l_T)\big)+d\big(F(l_T),F(y)\big)\\
    &\leq K(|x|+a_T+|l_T-y|)=Kd(x,y).
\end{align*}
Hence, $F$ is $K$-Lipschitz.
\end{proof}

\subsection{Lipschitz functions between threads with gaps.}
In a thread $T$, we say that a non-trivial open interval $(x,y)\subset \mathbb{R}$ is a \emph{gap} of $T$ if $x,y\in T$ and $(x,y)\cap T=\emptyset$. The points $x,y$ of a gap $C=(x,y)$ in a thread $T$ are called the \emph{endpoints} of $C$, and the value $d(x,y)$ is the \emph{length} of the gap. It is readily seen that a closed subset of $\mathbb{R}$ can have at most countably many distinct gaps. Hence, given any complete thread $T\subset\mathbb{R}$, we may consider the sequence $\{C^T_k\}_{k\in\mathbb{N}}$ of gaps in $T$. Moreover, since every thread $T$ is bounded, its sequence of gaps can be ordered so that $\text{length}(C^T_{k+1})\leq\text{length}(C^T_k)$ for all $k\in\mathbb{N}$. 

We are going to study in detail the behavior of Lipschitz maps between threads with infinitely many gaps.  We  have the following property.

\begin{lemma}
\label{lipconstinthreadswithnogaps}
Let $T$ and $S$ be two threads of length $l_T,l_S$ and width $a_T,a_S$ respectively. Let $K\geq 1$, and suppose that there is no gap in $T$ with length greater than or equal to $a_S/K$. Then for every $K$-Lipschitz function $F\colon T\rightarrow S$ we have that 

$$|F(q)-F(p)|\leq K |q-p|,\qquad\text{ for all }p,q\in T. $$
\end{lemma}
\begin{proof}
For any pair of points $p,q\in T$ with $p\leq q$ there exists an increasing finite sequence $(x_k)_{k=1}^n\subset T$ with $x_1=p$ and $x_n=q$ such that $d(x_k,x_{k+1})< a_S/K$ for all $1\leq k\leq n-1$. This implies that $d\big(F(x_{k+1}),F(x_k)\big)=\big|F(x_{k+1})-F(x_k)\big|$.

Hence, we have 
\begin{align*}
    |F(q)-F(p)|&\leq \sum_{k=1}^n\big|F(x_{k+1})-F(x_k)\big|=\sum_{k=1}^nd\big(F(x_{k+1}),F(x_k)\big)\\
    &\leq K \sum_{k=1}^nd(x_{k+1},x_k)\leq K\sum_{k=1}^n (x_{k+1}-x_k)=K(q-p).
\end{align*}
The result is proven.

\end{proof}

Next, we are going to prove an elementary proposition which will allow us to assume without loss of generality that the Lipschitz maps we consider are non-decreasing.
\begin{proposition}
\label{wlogFisnondecreasing}
Let $K\geq 1$. Let $T,S$ be two threads of length $l_T, l_S$ and width $a_T,a_S$ respectively, and let $F\colon T\rightarrow S$ be a $K$-Lipschitz function such that $F(0)=0$ and $F(l_T)=l_S$. Then there exists a non-decreasing Lipschitz function $\widehat{F}\colon T\rightarrow S$ with $\|\widehat{F}\|_\text{Lip}\leq\|F\|_\text{Lip}$ such that $\widehat{F}(0)=0$ and $\widehat{F}(l_T)=l_S$.
\end{proposition}
\begin{proof}
Put $K=\|F\|_\text{Lip}$. Notice that if $T$ has a gap $(p,q)$ of length greater than or equal to $a_S/K$, then the result follows directly putting $\widehat{F}(x)=0$ if $x\leq p$, and $\widehat{F}(x)=l_S$ if $x\geq q$. Suppose then that there are no gaps in $T$ with length greater than or equal to $a_S/K$.  

Now define $\widehat{F}\colon T\rightarrow S$ by
$$ \widehat{F}(x)=\max_{y\leq x} F(y).$$

Clearly, $\widehat{F}$ is non-decreasing with $F\leq \widehat{F}$, $\widehat{F}(0)=0$ and $\widehat{F}(l_T)=l_S$. It only remains to see that $\|\widehat{F}\|_\text{Lip}\leq K$. Using Proposition \ref{domainthreadisinterval}, we only need to prove that given $p,q\in T$ with $p\leq q$, we have
\begin{equation}
\label{NormalsegmentLipConstant}
d\big(\widehat{F}(q),\widehat{F}(p)\big)\leq K(q-p).
\end{equation}

Observe that $\widehat{F}(q)=F(z)$ for some $z\leq q$. If $z\leq p$ we necessarily have that $\widehat{F}(q)=\widehat{F}(p)$ and the equation is trivially verified. Otherwise, using Lemma \ref{lipconstinthreadswithnogaps} with $p\leq z$ yields
\begin{align*}
    d\big(\widehat{F}(q),\widehat{F}(p)\big)&\leq \widehat{F}(q)-\widehat{F}(p)=F(z)-\widehat{F}(p)\\
    &\leq F(z)-F(p)\leq K(z-p)\leq K(q-p).
\end{align*}
and equation \eqref{NormalsegmentLipConstant} is proven.

We conclude that $\|\widehat{F}\|_\text{Lip}\leq K$ and the result is proven. 
\end{proof}

We can also use Lemma \ref{lipconstinthreadswithnogaps} to prove a similar result to the one above.
\begin{proposition}
\label{threadontosubinterval}
Let $T$ and $S$ be two threads with length $l_T$ and $l_S$, and width $a_T$ and $a_S$ respectively. Let $K\geq 1$. Suppose there exists a $K$-Lipschitz function $F\colon T\rightarrow S$ such that $F(0)=A$ and $F(l_T)=B$, for two points $A,B\in S$ with $A<B$. If $T$ does not have any gap of length greater than or equal to $a_S/K$, then the function $\widehat{F}\colon T\rightarrow [A,B]_S$ defined by 
$$
\widehat{F}(x)=
\begin{cases}
A,~&\text{ if }F(x)\leq A\\
F(x), ~&\text{ if }F(x)\in [A,B]_S\\
B,~ &\text{ if }F(x)\geq B
\end{cases}
$$
is $K$-Lipschitz as well.

\end{proposition}
\begin{proof}
As before, by Proposition \ref{domainthreadisinterval} we only need to check that for every $p,q\in T$ with $p\leq q$, we have:
$$d\big(\widehat{F}(q),\widehat{F}(p)\big)\leq K(q-p). $$
We will only prove the case when $F(p)\leq A$ and $F(q)\in [A,B]_S$, since the remaining possibilities are shown similarly. By Lemma \ref{lipconstinthreadswithnogaps}, we have in this case that
\begin{align*}
    d\big(\widehat{F}(p),\widehat{F}(q)\big)&=d\big(A,F(q)\big)\leq F(q)-A\\
    &\leq F(q)-F(p)\leq K(q-p).
\end{align*}
We conclude that $\|\widehat{F}\|_\text{Lip}\leq K$.

\end{proof}

Let us now give some definitions and prove some technical results which will be heavily used in the proof of the main theorem of the section. Let $T$ and $S$ be two threads, and suppose there is a Lipschitz function $F\colon T\rightarrow S$ which is non-decreasing. We say that a gap $(p_T,q_T)$ in $T$ \textit{jumps over a gap $(p_S,q_S)$ in $S$ with respect to $F$} if $F(p_T)\leq p_S$ and $F(q_T)\geq q_S$ (see Figure \ref{examplegapjumpthread}).

The first lemma we prove says intuitively that if we have a non-decreasing Lipschitz function $F$ between two threads $T$ and $S$ that fixes the extreme points of the threads, then every gap in $S$ must be jumped by a gap in $T$ with respect to $F$. Although this result is fairly intuitive, we include the (simple) proof for completeness.

\begin{lemma}
\label{gapjumpsovergap}
Let $T$ and $S$ be two threads of length $l_T$ and $l_S$ respectively. Suppose that there is a non-decreasing Lipschitz function $F\colon T\rightarrow S$ such that $F(0)=0$ and $F(l_T)=l_S$. Let $C^S$ be a gap in $S$. Then there exists a gap in $T$ that jumps over $C^S$ with respect to $F$.
\end{lemma}
\begin{proof}
Define $p_S,q_S\in S$ such that $C^S=(p_S,q_S)$. Consider the points:
\begin{align*}
    e_-&=\max\{F(x)\in S\colon~x\in T,~F(x)\leq p_S\},\\
    e_+&=\min\{F(y)\in S\colon~y\in T,~F(y)\geq q_S\}. 
\end{align*}
These minimum and maximum values always exist since we have that $F(0)=0$ and $F(l_T)=l_S$, and $T$ is compact. Hence, we can find 
\begin{align*}
    x_-&=\max\{x\in T\colon~F(x)=e_-\},\\
    y_+&=\min\{y\in T\colon~F(y)=e_+\}.
\end{align*}

Since $F$ is non-decreasing and $e_-<e_+$, we have that $x_-<y_+$ and $(x_-,y_+)_T=\emptyset$. Moreover, both $x_-$ and $y_+$ belong to $T$ again by compactness, so $(x_-,y_+)$ is a gap in $T$. The gap $(x_-,y_+)$ jumps over $(p_S,q_S)$ with respect to $F$.
\end{proof}

\begin{figure}
\centering

\tikzset{every picture/.style={line width=0.75pt}} 

\begin{tikzpicture}[x=0.75pt,y=0.75pt,yscale=-1,xscale=1]

\draw  [draw opacity=0] (475.12,106.99) .. controls (493.08,112.59) and (509.52,123.72) .. (521.56,139.91) .. controls (545.46,172.05) and (544.77,214.37) .. (522.82,244.74) -- (448.93,192.13) -- cycle ; \draw  [color={rgb, 255:red, 0; green, 0; blue, 0 }  ,draw opacity=1 ] (475.12,106.99) .. controls (493.08,112.59) and (509.52,123.72) .. (521.56,139.91) .. controls (545.46,172.05) and (544.77,214.37) .. (522.82,244.74) ;
\draw    (231.47,114.16) -- (226.89,121.12) ;
\draw    (157.81,99.43) -- (159.89,107.61) ;
\draw  [draw opacity=0] (108.19,245.64) .. controls (107.89,245.22) and (107.59,244.79) .. (107.3,244.36) .. controls (79.19,203.13) and (88.91,146.33) .. (129.03,117.48) .. controls (138.29,110.82) and (148.35,106.24) .. (158.71,103.63) -- (179.93,192.13) -- cycle ; \draw  [color={rgb, 255:red, 0; green, 0; blue, 0 }  ,draw opacity=1 ] (108.19,245.64) .. controls (107.89,245.22) and (107.59,244.79) .. (107.3,244.36) .. controls (79.19,203.13) and (88.91,146.33) .. (129.03,117.48) .. controls (138.29,110.82) and (148.35,106.24) .. (158.71,103.63) ;
\draw  [draw opacity=0] (228.87,117.66) .. controls (237.79,123.46) and (245.85,130.9) .. (252.55,139.91) .. controls (276.46,172.05) and (275.77,214.37) .. (253.81,244.74) -- (179.93,192.13) -- cycle ; \draw   (228.87,117.66) .. controls (237.79,123.46) and (245.85,130.9) .. (252.55,139.91) .. controls (276.46,172.05) and (275.77,214.37) .. (253.81,244.74) ;
\draw  [fill={rgb, 255:red, 0; green, 0; blue, 0 }  ,fill opacity=1 ] (106.17,245.63) .. controls (106.17,244.49) and (107.08,243.56) .. (108.19,243.56) .. controls (109.31,243.56) and (110.21,244.49) .. (110.21,245.63) .. controls (110.21,246.78) and (109.31,247.71) .. (108.19,247.71) .. controls (107.08,247.71) and (106.17,246.78) .. (106.17,245.63) -- cycle ;
\draw  [fill={rgb, 255:red, 0; green, 0; blue, 0 }  ,fill opacity=1 ] (251.87,244.63) .. controls (251.87,243.48) and (252.78,242.55) .. (253.89,242.55) .. controls (255.01,242.55) and (255.91,243.48) .. (255.91,244.63) .. controls (255.91,245.77) and (255.01,246.7) .. (253.89,246.7) .. controls (252.78,246.7) and (251.87,245.77) .. (251.87,244.63) -- cycle ;
\draw  [color={rgb, 255:red, 208; green, 2; blue, 27 }  ,draw opacity=1 ][fill={rgb, 255:red, 208; green, 2; blue, 27 }  ,fill opacity=1 ] (226.98,117.74) .. controls (226.98,116.6) and (227.88,115.67) .. (229,115.67) .. controls (230.11,115.67) and (231.02,116.6) .. (231.02,117.74) .. controls (231.02,118.89) and (230.11,119.82) .. (229,119.82) .. controls (227.88,119.82) and (226.98,118.89) .. (226.98,117.74) -- cycle ;
\draw  [color={rgb, 255:red, 208; green, 2; blue, 27 }  ,draw opacity=1 ][fill={rgb, 255:red, 208; green, 2; blue, 27 }  ,fill opacity=1 ] (156.83,103.59) .. controls (156.83,102.44) and (157.74,101.52) .. (158.85,101.52) .. controls (159.97,101.52) and (160.87,102.44) .. (160.87,103.59) .. controls (160.87,104.73) and (159.97,105.66) .. (158.85,105.66) .. controls (157.74,105.66) and (156.83,104.73) .. (156.83,103.59) -- cycle ;
\draw  [dash pattern={on 4.5pt off 4.5pt}]  (121.47,245.43) -- (238.24,245.43) ;
\draw [shift={(241.24,245.43)}, rotate = 180] [fill={rgb, 255:red, 0; green, 0; blue, 0 }  ][line width=0.08]  [draw opacity=0] (5.36,-2.57) -- (0,0) -- (5.36,2.57) -- cycle    ;
\draw [shift={(118.47,245.43)}, rotate = 0] [fill={rgb, 255:red, 0; green, 0; blue, 0 }  ][line width=0.08]  [draw opacity=0] (5.36,-2.57) -- (0,0) -- (5.36,2.57) -- cycle    ;
\draw  [draw opacity=0] (377.25,245.6) .. controls (376.94,245.19) and (376.62,244.77) .. (376.31,244.35) .. controls (346.54,204.32) and (354.92,148.49) .. (395.02,119.65) .. controls (402.93,113.96) and (411.48,109.75) .. (420.34,106.95) -- (448.93,192.13) -- cycle ; \draw   (377.25,245.6) .. controls (376.94,245.19) and (376.62,244.77) .. (376.31,244.35) .. controls (346.54,204.32) and (354.92,148.49) .. (395.02,119.65) .. controls (402.93,113.96) and (411.48,109.75) .. (420.34,106.95) ;
\draw  [fill={rgb, 255:red, 0; green, 0; blue, 0 }  ,fill opacity=1 ] (375.23,245.6) .. controls (375.23,244.45) and (376.13,243.52) .. (377.25,243.52) .. controls (378.37,243.52) and (379.27,244.45) .. (379.27,245.6) .. controls (379.27,246.74) and (378.37,247.67) .. (377.25,247.67) .. controls (376.13,247.67) and (375.23,246.74) .. (375.23,245.6) -- cycle ;
\draw  [fill={rgb, 255:red, 0; green, 0; blue, 0 }  ,fill opacity=1 ] (520.88,244.63) .. controls (520.88,243.48) and (521.78,242.55) .. (522.9,242.55) .. controls (524.01,242.55) and (524.92,243.48) .. (524.92,244.63) .. controls (524.92,245.77) and (524.01,246.7) .. (522.9,246.7) .. controls (521.78,246.7) and (520.88,245.77) .. (520.88,244.63) -- cycle ;
\draw  [color={rgb, 255:red, 208; green, 2; blue, 27 }  ,draw opacity=1 ][fill={rgb, 255:red, 208; green, 2; blue, 27 }  ,fill opacity=1 ] (515.6,134.01) .. controls (515.6,132.87) and (516.51,131.94) .. (517.62,131.94) .. controls (518.74,131.94) and (519.64,132.87) .. (519.64,134.01) .. controls (519.64,135.16) and (518.74,136.08) .. (517.62,136.08) .. controls (516.51,136.08) and (515.6,135.16) .. (515.6,134.01) -- cycle ;
\draw  [color={rgb, 255:red, 208; green, 2; blue, 27 }  ,draw opacity=1 ][fill={rgb, 255:red, 208; green, 2; blue, 27 }  ,fill opacity=1 ] (383.86,127.32) .. controls (383.86,126.17) and (384.76,125.25) .. (385.88,125.25) .. controls (386.99,125.25) and (387.9,126.17) .. (387.9,127.32) .. controls (387.9,128.47) and (386.99,129.39) .. (385.88,129.39) .. controls (384.76,129.39) and (383.86,128.47) .. (383.86,127.32) -- cycle ;
\draw    (289.46,179.85) -- (336.4,179.85) ;
\draw [shift={(339.4,179.85)}, rotate = 180] [fill={rgb, 255:red, 0; green, 0; blue, 0 }  ][line width=0.08]  [draw opacity=0] (7.14,-3.43) -- (0,0) -- (7.14,3.43) -- cycle    ;
\draw  [dash pattern={on 4.5pt off 4.5pt}]  (393.65,245.34) -- (510.43,245.34) ;
\draw [shift={(513.43,245.34)}, rotate = 180] [fill={rgb, 255:red, 0; green, 0; blue, 0 }  ][line width=0.08]  [draw opacity=0] (5.36,-2.57) -- (0,0) -- (5.36,2.57) -- cycle    ;
\draw [shift={(390.65,245.34)}, rotate = 0] [fill={rgb, 255:red, 0; green, 0; blue, 0 }  ][line width=0.08]  [draw opacity=0] (5.36,-2.57) -- (0,0) -- (5.36,2.57) -- cycle    ;
\draw    (419.24,102.84) -- (421.74,111.02) ;
\draw    (475.84,103.25) -- (473.76,111.02) ;

\draw (151.6,80.54) node [anchor=north west][inner sep=0.75pt]  [font=\large]  {$p$};
\draw (234.07,97) node [anchor=north west][inner sep=0.75pt]  [font=\large]  {$q$};
\draw (92.63,249.49) node [anchor=north west][inner sep=0.75pt]  [font=\large]  {$0$};
\draw (257.08,246.99) node [anchor=north west][inner sep=0.75pt]  [font=\large]  {$l_{T}$};
\draw (169.72,223.89) node [anchor=north west][inner sep=0.75pt]  [font=\large]  {$a_{T}$};
\draw (348.93,103.19) node [anchor=north west][inner sep=0.75pt]  [font=\large]  {$F( p)$};
\draw (520.34,114.25) node [anchor=north west][inner sep=0.75pt]  [font=\large]  {$F( q)$};
\draw (363.47,249.49) node [anchor=north west][inner sep=0.75pt]  [font=\large]  {$0$};
\draw (524.07,244.92) node [anchor=north west][inner sep=0.75pt]  [font=\large]  {$l_{S}$};
\draw (445.47,223.63) node [anchor=north west][inner sep=0.75pt]  [font=\large]  {$a_{S}$};
\draw (308.52,160.75) node [anchor=north west][inner sep=0.75pt]  [font=\large]  {$F$};
\draw (409.66,86.03) node [anchor=north west][inner sep=0.75pt]  [font=\large]  {$x$};
\draw (473.92,86.14) node [anchor=north west][inner sep=0.75pt]  [font=\large]  {$y$};

\end{tikzpicture}

\caption{The gap $(p,q)$ jumps over $(x,y)$ with respect to $F$.}
\label{examplegapjumpthread}
\end{figure}
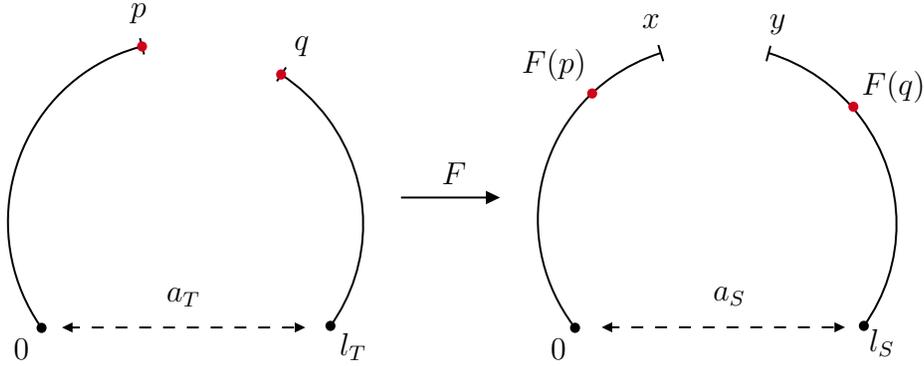

The second lemma we prove can also be easily deduced and is intuitively clear. It shows that if a small enough gap (in a sense that is made explicit in the statement of the lemma) $C^T$ in a thread $T$ jumps over several gaps in a thread $S$ simultaneously with respect to a Lipschitz function $F$, then the length of $C^T$ must be bigger than the length of the smallest subinterval of $[0,1]$ that contains all the gaps $C^T$ jumps over, divided by the Lipschitz constant of $F$. 
\begin{lemma}
\label{onlylonggapscanjump}
Let $K>1$, and let $T$ and $S$ be two threads of length $l_T$ and $l_S$ respectively. Denote by $a_S$ the width of $S$. Suppose that there is a non-decreasing Lipschitz function $F\colon T\rightarrow S$ with $\|F\|_\text{Lip}= K$ such that $F(0)=0$ and $F(l_T)=l_S$. Let $C^T$ be a gap in $T$ such that $\text{length}(C^T)<a_S/K$, and let $\big((x_j,y_j)\big)_{j=1}^k$ be a finite collection of different gaps in $S$. If $C^T$ jumps over $(x_j,y_j)$ with respect to $F$ for all $1\leq j\leq k$, then 

$$K\cdot\text{length} (C^T)\geq \max_{j\neq j'} |y_j-x_{j'}|.$$
\end{lemma}
\begin{proof}
Put $C^T=(p,q)$ with $p,q\in T$. Since $C^T$ jumps over $(x_j,y_j)$ for all $1\leq j\leq k$, we have that $F(p)\leq x_j$ and $F(q)\geq y_j$. Hence, we have that 
$$F(q)-F(p)\geq \max_{j\neq j'} |y_j-x_{j'}|.$$

Since $d(p,q)<a_S/K$, when computing the distance between $F(p)$ and $F(q)$ in the thread $S$, we necessarily have that $d\big(F(p),F(q)\big)= F(q)-F(p)$. Therefore, applying that $F$ is $K$-Lipschitz we obtain:
$$\max_{j\neq j'} |y_j-x_{j'}|\leq d(F(q),F(p))\leq K\cdot \text{length}(C^T),$$
and the result is proven.
\end{proof}

We are going to define also a particular kind of intervals which will be useful in the proof of Theorem \ref{nolipschitzmap}. Let $(a,b)\subset [0,1]$ be a nontrivial open interval, and let $r>0$. We define the \textit{sweeping of $[a,b]$ by $r$} as the interval
$$ \mathcal{D}_r(a,b)=(b-r,a+r).$$

Notice that if $r\leq (b-a)/2$, then $\mathcal{D}_r(a,b)=\emptyset$. We can prove two simple properties about this concept.

\begin{proposition}
\label{propertiesdilations}
The following properties are verified:

\begin{enumerate}[label=(\arabic*), ref=(\arabic*)]
    \item Let $(a,b)\subset [0,1]$ be a nontrivial open interval, and let $r>0$. Then the Lebesgue measure of the sweeping $\mathcal{D}_r(a,b)$ is less than $2r$.
    \item Let $r>0$, and let $T$ and $S$ be threads of length $l_T$ and $l_S$ respectively. Let $F\colon T\rightarrow S$ be a non-decreasing $K$-Lipschitz map such that $F(0)=0$ and $F(l_T)=l_S$. Suppose that there is a gap $C^T$ in $T$ such that $\text{length}(C^T)<a_S/K$, and such that $C^T$ jumps over two gaps $C^S_1,C^S_2$ in $S$ with respect to $F$. Moreover, suppose that $C^S_2\nsubseteq \mathcal{D}_r(C^S_1)$. Then $K\cdot\text{length}(C^T)>r$. 
\end{enumerate}
\end{proposition}
\begin{proof}
Statement $(1)$ is easy to see.
For statement $(2)$, put $C^S_1=(x_1,y_1),~ C^S_2=(x_2,y_2)$ with $x_1,x_2,y_1,y_2\in S$. Notice that if $C^S_2\nsubseteq \mathcal{D}_r(C^S_1)$, this means that either $y_1-r-x_2>0$, or $y_2-x_1-r>0$. In any case, we obtain that 

$$ \max \{|y_1-x_2|,|y_2-x_1|\}>r,$$
Now, since $C^T$ jumps over $C^S_1$ and $C^S_2$ simultaneously and $\text{length}(C^T)<a_S/K$, the result follows from Lemma \ref{onlylonggapscanjump}. 
\end{proof}

\begin{figure}

\tikzset{every picture/.style={line width=0.75pt}} 

\begin{tikzpicture}[x=0.75pt,y=0.75pt,yscale=-1,xscale=1]

\draw  [draw opacity=0] (361.63,59.95) .. controls (394.32,65.78) and (424.81,84.06) .. (445.38,113.43) .. controls (477.16,158.81) and (475.95,217.42) .. (446.76,260.76) -- (338.89,187.99) -- cycle ; \draw   (361.63,59.95) .. controls (394.32,65.78) and (424.81,84.06) .. (445.38,113.43) .. controls (477.16,158.81) and (475.95,217.42) .. (446.76,260.76) ;
\draw  [color={rgb, 255:red, 74; green, 76; blue, 226 }  ,draw opacity=1 ] (418.33,88.02) -- (426.53,88.02)(422.43,84.03) -- (422.43,92) ;
\draw  [color={rgb, 255:red, 74; green, 76; blue, 226 }  ,draw opacity=1 ] (255.67,84.68) -- (263.87,84.68)(259.77,80.7) -- (259.77,88.67) ;
\draw  [draw opacity=0] (276.11,74.03) .. controls (290.28,66.2) and (305.43,61.29) .. (320.78,59.15) -- (338.89,187.99) -- cycle ; \draw   (276.11,74.03) .. controls (290.28,66.2) and (305.43,61.29) .. (320.78,59.15) ;
\draw    (361.83,56.33) -- (360.07,64.1) ;
\draw    (321.17,55) -- (322.07,63.43) ;
\draw [color={rgb, 255:red, 208; green, 2; blue, 27 }  ,draw opacity=1 ]   (273.83,71) -- (278.5,78.33) ;
\draw   (227.96,260.41) .. controls (227.96,258.77) and (229.29,257.44) .. (230.92,257.44) .. controls (232.56,257.44) and (233.88,258.77) .. (233.88,260.41) .. controls (233.88,262.04) and (232.56,263.37) .. (230.92,263.37) .. controls (229.29,263.37) and (227.96,262.04) .. (227.96,260.41) -- cycle ;
\draw   (443.8,260.76) .. controls (443.8,259.12) and (445.13,257.8) .. (446.76,257.8) .. controls (448.4,257.8) and (449.73,259.12) .. (449.73,260.76) .. controls (449.73,262.39) and (448.4,263.72) .. (446.76,263.72) .. controls (445.13,263.72) and (443.8,262.39) .. (443.8,260.76) -- cycle ;
\draw  [dash pattern={on 4.5pt off 4.5pt}]  (253,260) -- (427,260) ;
\draw [shift={(430,260)}, rotate = 180] [fill={rgb, 255:red, 0; green, 0; blue, 0 }  ][line width=0.08]  [draw opacity=0] (5.36,-2.57) -- (0,0) -- (5.36,2.57) -- cycle    ;
\draw [shift={(250,260)}, rotate = 0] [fill={rgb, 255:red, 0; green, 0; blue, 0 }  ][line width=0.08]  [draw opacity=0] (5.36,-2.57) -- (0,0) -- (5.36,2.57) -- cycle    ;
\draw  [draw opacity=0] (271.8,91.65) .. controls (296.07,74.74) and (324.24,67.76) .. (351.3,69.9) -- (338.89,187.99) -- cycle ; \draw  [color={rgb, 255:red, 74; green, 76; blue, 226 }  ,draw opacity=1 ] (271.8,91.65) .. controls (296.07,74.74) and (324.24,67.76) .. (351.3,69.9) ;
\draw [color={rgb, 255:red, 74; green, 76; blue, 226 }  ,draw opacity=1 ]   (360,70) ;
\draw [shift={(360,70)}, rotate = 180] [fill={rgb, 255:red, 74; green, 76; blue, 226 }  ,fill opacity=1 ][line width=0.08]  [draw opacity=0] (8.4,-2.1) -- (0,0) -- (8.4,2.1) -- cycle    ;
\draw [color={rgb, 255:red, 74; green, 76; blue, 226 }  ,draw opacity=1 ]   (269.4,93.25) -- (268.73,93.68) -- (268.73,93.68) ;
\draw [shift={(269.06,93.46)}, rotate = 327.4] [fill={rgb, 255:red, 74; green, 76; blue, 226 }  ,fill opacity=1 ][line width=0.08]  [draw opacity=0] (8.4,-2.1) -- (0,0) -- (8.4,2.1) -- cycle    ;
\draw [color={rgb, 255:red, 208; green, 2; blue, 27 }  ,draw opacity=1 ]   (222.5,121) -- (229.83,125.67) ;
\draw  [draw opacity=0] (230.92,260.41) .. controls (202.45,217.99) and (202.21,164.63) .. (226.21,122.94) -- (338.89,187.99) -- cycle ; \draw   (230.92,260.41) .. controls (202.45,217.99) and (202.21,164.63) .. (226.21,122.94) ;
\draw  [draw opacity=0][dash pattern={on 4.5pt off 4.5pt}] (259.89,84.62) .. controls (261.32,83.53) and (262.78,82.46) .. (264.27,81.42) .. controls (313.55,46.91) and (378.48,51.38) .. (422.48,88.38) -- (338.89,187.99) -- cycle ; \draw  [color={rgb, 255:red, 74; green, 76; blue, 226 }  ,draw opacity=1 ][dash pattern={on 4.5pt off 4.5pt}] (259.89,84.62) .. controls (261.32,83.53) and (262.78,82.46) .. (264.27,81.42) .. controls (313.55,46.91) and (378.48,51.38) .. (422.48,88.38) ;
\draw  [fill={rgb, 255:red, 65; green, 117; blue, 5 }  ,fill opacity=1 ] (217.59,134.62) .. controls (217.59,132.99) and (218.92,131.66) .. (220.55,131.66) .. controls (222.19,131.66) and (223.51,132.99) .. (223.51,134.62) .. controls (223.51,136.26) and (222.19,137.59) .. (220.55,137.59) .. controls (218.92,137.59) and (217.59,136.26) .. (217.59,134.62) -- cycle ;
\draw  [fill={rgb, 255:red, 65; green, 117; blue, 5 }  ,fill opacity=1 ] (378.92,65.29) .. controls (378.92,63.66) and (380.25,62.33) .. (381.89,62.33) .. controls (383.52,62.33) and (384.85,63.66) .. (384.85,65.29) .. controls (384.85,66.93) and (383.52,68.25) .. (381.89,68.25) .. controls (380.25,68.25) and (378.92,66.93) .. (378.92,65.29) -- cycle ;

\draw (308,35.4) node [anchor=north west][inner sep=0.75pt]    {$x_{1}$};
\draw (361.33,34.07) node [anchor=north west][inner sep=0.75pt]    {$y_{1}$};
\draw (216.39,262.78) node [anchor=north west][inner sep=0.75pt]    {$0$};
\draw (456.39,260.38) node [anchor=north west][inner sep=0.75pt]    {$l$};
\draw (338,233.4) node [anchor=north west][inner sep=0.75pt]    {$a$};
\draw (312.67,76.5) node [anchor=north west][inner sep=0.75pt]  [color={rgb, 255:red, 74; green, 76; blue, 226 }  ,opacity=1 ]  {$r$};
\draw (204,103.17) node [anchor=north west][inner sep=0.75pt]  [color={rgb, 255:red, 208; green, 2; blue, 27 }  ,opacity=1 ]  {$x_{2}$};
\draw (264.67,49.83) node [anchor=north west][inner sep=0.75pt]  [color={rgb, 255:red, 208; green, 2; blue, 27 }  ,opacity=1 ]  {$y_{2}$};
\draw (181.89,124.02) node [anchor=north west][inner sep=0.75pt]  [color={rgb, 255:red, 65; green, 117; blue, 5 }  ,opacity=1 ]  {$F( p)$};
\draw (382.55,43.36) node [anchor=north west][inner sep=0.75pt]  [color={rgb, 255:red, 65; green, 117; blue, 5 }  ,opacity=1 ]  {$F( q)$};
\draw (222.67,70.5) node [anchor=north west][inner sep=0.75pt]  [font=\footnotesize,color={rgb, 255:red, 74; green, 76; blue, 226 }  ,opacity=1 ]  {$y_{1} -r$};
\draw (428.67,74.5) node [anchor=north west][inner sep=0.75pt]  [font=\footnotesize,color={rgb, 255:red, 74; green, 76; blue, 226 }  ,opacity=1 ]  {$x_{1} +r$};

\end{tikzpicture}

\caption{The gap $C^T=(p,q)$ jumps over $(x_1,y_1)$ and $(x_2,y_2)$, and $(x_2,y_2)\nsubseteq\mathcal{D}_r\big((x_1,y_1)\big)=(y_1-r,x_1+r)$. Hence, $K\cdot \text{length}(C^T)\geq r$.}
\label{drawingpropsweeping}
\end{figure}
In figure \ref{drawingpropsweeping} we have a representation of the situation in $(2)$ of the previous result. 

Finally we are able to prove the main result of the section:

\begin{theorem}
\label{nolipschitzmap}
Let $K_0\geq 1$, let $\{T_n\}_{n\in\mathbb{N}}$ be a countable family of threads of length $l_n$ and width $a_n$ respectively for each $n\in\mathbb{N}$, and let $\varepsilon>0$ be such that for every $n\in\mathbb{N}$:
\begin{itemize}
    \item The Lebesgue measure of $T_n$ is bigger than $\varepsilon$.
    \item If $I$ is an open interval such that $I\cap T_n$ is nonempty, then there exist infinitely many gaps of $T_n$ contained in $I$.
\end{itemize}

Then, there exists a decreasing sequence $\gamma^*=(\gamma_k^*)_{k\in\mathbb{N}}$ of positive real numbers with the following property:

Let $1\leq K'\leq K$, let $S$ be a thread of length $l_S$, and let $\{C^S_k\}_{k\in\mathbb{N}}$ be the sequence of gaps of $S$ ordered decreasingly according to their length. If $\text{length}(C^S_{k})\leq\gamma^*_k$ for all $k\in\mathbb{N}$, then for every $n\in\mathbb{N}$ such that $\text{length}(C^S_1)<a_n/K'$ there does not exist any $K'$-Lipschitz function $S\rightarrow T_n$ such that $F(0)=0$ and $F(l_S)=l_n$.
\end{theorem}

\begin{proof}

For each $n\in\mathbb{N}$, let $\{G^n_k\}_{k\in\mathbb{N}}$ be the sequence of gaps of $T_n$ ordered decreasingly according to their length, and put $\alpha^n_k=\text{length}(G^n_k)$. Then $\alpha^n=(\alpha^n_k)_{k\in\mathbb{N}}$ is the decreasing sequence of lengths of the gaps in the thread $T_n$ for each $n\in\mathbb{N}$.

We are going to construct inductively by a diagonal method the sequence $\gamma^*=(\gamma_k^*)_{k\in\mathbb{N}}$ with the following properties: 

\begin{enumerate}[label=(\arabic*), ref=(\arabic*)]
    \item $\gamma^*_k< 2^{-(k+1)}K^{-1}\varepsilon $, for all $k\in\mathbb{N}$.
    \item Let $1\leq K'\leq K$, let $S$ be a thread, and let $\{C^S_i\}_{i\in\mathbb{N}}$ be the sequence of gaps of $S$ ordered decreasingly according to their length. If $\text{length}(C^S_{i})\leq \gamma^*_i$ for all $i\leq k$, then for every $i\in\mathbb{N}$ with $i\leq k$ such that $\text{length}(C^S_1)< a_i/K$ there is no Lipschitz map $F\colon S\rightarrow T_i$ with $\|F\|_\text{Lip}\leq K'$ such that $F(0)=0$ and $F(l_S)=l_i$. 
\end{enumerate}

For $k=1$, define $\gamma^*_1$ such that $0<\gamma_1^*<\min\{2^{-2}K^{-1}\varepsilon,K^{-1}\alpha_1^1\}$. Property $(1)$ above is verified. Let $1\leq K'\leq K$, let $S$ be a thread, let $\{C^S_i\}_{i\in\mathbb{N}}$ be the sequence of its gaps in decreasing length order, and suppose that $\text{length}(C^S_{1})\leq \gamma^*_1$. Suppose by contradiction that $\text{length}(C^S_1)<a_1/K'$, and that there exists a Lipschitz map $F\colon S\rightarrow T_1$ with $\|F\|_\text{Lip}\leq K'$ and $F(0)=0$ and $F(l_S)=l_1$. We assume $F$ to be non-decreasing by Proposition \ref{wlogFisnondecreasing}.

The thread $T_1$ has the gap $G^1_1$ of length $\alpha^1_1$. By Lemma \ref{gapjumpsovergap}, there exists $i\in\mathbb{N}$ such that the gap $C^S_i$ in $S$ jumps over $G^1_1$. Then, by Lemma \ref{onlylonggapscanjump}, since $\text{length}(C^S_i)\leq \text{length}(C^S_1)<a_1/K'$, we have that $K'\cdot \text{length}(C^S_i)>\alpha^1_1$. However, we know that 
$$K'\cdot \text{length}(C^S_i)\leq K\cdot \text{length}(C^S_1)\leq K\gamma^*_1<\alpha_1^1,$$
a contradiction. The first step of the induction is done.

Suppose we have selected $\{\gamma^*_i\}_{i=1}^k$ verifying the desired properties for $k\in \mathbb{N}$. Before continuing with the proof, let us informally give some intuition of the technical argument that follows. We want to define the next element in the sequence (that is, $\gamma^*_{k+1}$) small enough so that any thread $S$ with gaps smaller than the first $k+1$ elements of $\gamma^*$, and smaller than $a_{k+1}$, cannot be mapped with a $K$-Lipschitz function that preserves the extremes into the thread $T_{k+1}$. However, since the first $k$ elements of $\gamma^*$ are already set and do not depend on the next thread $T_{k+1}$, the first $k$ gaps in $S$ could jump over many gaps in $T_{k+1}$, and in many different ways. Nevertheless, as we are going to see, the fact that $T_{k+1}$ is of large enough measure and contains infinitely many gaps in each intersecting open interval, and the way in which we chose the first $k$ elements in $\gamma^*$, ensures that there will always be infinitely many gaps in $T_{k+1}$ that cannot be jumped over in any way with the first $k$ gaps of $S$ with any suitable $K$-Lipschitz function. Hence, we will be able to define $\gamma^{k+1}$ small enough so that the biggest of these ``unjumped" gaps in $T_{k+1}$ cannot be jumped over either by the remaining gaps of $S$. We just need to account for all the possibilities in which the first $k$ gaps of $S$ might behave under a $K$-Lipschitz function. 

Let $\sigma=\{j_i\}_{i=1}^k$ be an ordering of the sequence $\{1,\dots,k\}$. For convenience of the notation, put $j_0=0$, and $n_{j_0}=1$, and define
$$D_{(j_0,j_1)}=\mathcal{D}_{K\gamma^*_{j_1}}(G^{k+1}_{n_{j_0}}). $$
$D_{(j_0,j_1)}$ is the sweeping of $G^{k+1}_{n_{j_0}}$, the biggest gap of $T_{k+1}$, by $K\gamma^*_{j_1}$. The measure of $D_{(j_0,j_1)}$ is at most $2K\gamma_{j_1}^* <2^{-j_1}\varepsilon< \varepsilon$. Hence, since $T_{k+1}$ has measure greater than $\varepsilon$, the set $T_{k+1}\setminus D_{(j_0,j_1)}=T_{k+1}\cap ([0,l_{k+1}]\setminus D_{(j_0,j_1)})$ is nonempty. Since $[0,l_{k+1}]\setminus D_{(j_0,j_1)}$ is a finite union of open intervals, by hypothesis there must exist infinitely many gaps in $T_{k+1}\setminus D_{(j_0,j_1)}$. We can then consider
$$n_{(j_0,j_1)}=\min\{n> n_{j_0}\colon ~G^{k+1}_n\nsubseteq D_{(j_0,j_1)}\}. $$

Intuitively, $G^{k+1}_{n_{(j_0,j_1)}}$ is the biggest gap of $T_{k+1}$ smaller than $G^{k+1}_{n_{j_0}}$ which is not contained in the sweeping $D_{(j_0,j_1)}$. We continue the process defining
$$ D_{(j_0,j_1,j_2)}=D_{(j_0,j_1)}\cup \mathcal{D}_{K\gamma^*_{j_2}}(G^{k+1}_{n_{(j_0,j_1)}}).$$

The measure of $D_{(j_0,j_1,j_2)}$ is at most $(2^{-j_1}+2^{-j_2})\varepsilon<\varepsilon$, and its complement in $[0,l_{k+1}]$ is still a finite union of open intervals, so by the properties of $T_{k+1}$ we can make the same argument as before to find
$$n_{(j_0,j_1,j_2)}=\min\{n> n_{(j_0,j_1)}\colon ~G^{k+1}_n\nsubseteq D_{(j_0,j_1,j_2)}\}, $$
which will be the biggest gap of $T_{k+1}$ smaller than $G^{k+1}_{n_{(j_0,j_1)}}$ not contained in $D_{(j_0,j_1,j_2)}$. Hence, it is not contained in neither the sweeping $\mathcal{D}_{K\gamma^*_{j_1}}(G^{k+1}_{n_{j_0}})$ nor $\mathcal{D}_{K\gamma^*_{j_2}}(G^{k+1}_{n_{(j_0,j_1)}})$. 

Repeating this process $k$ times, we can define $n_{\sigma}=n_{(j_0,\dots,j_{k-1})}\in \mathbb{N}$ such that $G_{n_{\sigma}}^{k+1}$ is the biggest gap of $T_{k+1}$ not contained in $\mathcal{D}_{K\gamma^*_{j_k}}(G^{k+1}_{n_{(j_0,\dots,j_{i-1})}})$ for any $1\leq i\leq k$, and smaller than $G^{k+1}_{n_{(j_0,\dots,j_{i-1})}}$ for every $1\leq i\leq k-1$. Notice that this last condition can be written as:
\begin{equation}
\label{gammasigmaisthesmallest}
\alpha^{k+1}_{n_\sigma}< \min_{1\leq i\leq k} \alpha^{k+1}_{n_{(j_0,\dots,j_{i-1})}}.
\end{equation}

Now, let $\Omega_k=\{\sigma=\{j_i\}_{i=1}^k\colon \sigma\text{ is and ordering of }\{1,\dots,k\}\}$. Clearly $\Omega_k$ is a finite set, so we can define $n_{\Omega_k}=\max\{n_\sigma\colon \sigma\in \Omega_k\}$. The corresponding gap $G^{k+1}_{n_{\Omega_k}}$ is smaller than or equal to each $G_{n_{\sigma}}^{k+1}$. Equivalently, we have that 
\begin{equation}
\label{gammaomegaisthesmallest}
\alpha^{k+1}_{n_{\Omega_k}}\leq \min_{\sigma\in \Omega_k}\alpha^{k+1}_{n_\sigma}.
\end{equation}

Finally, choose $\gamma^*_{k+1}$ so that $0<\gamma^{*}_{k+1}<\min\{2^{-{(k+2)}}K^{-1}\varepsilon,K^{-1}\alpha_{n_{\Omega_k}}^{k+1}\}$. Again, property $(1)$ of the induction is verified. Let $1\leq K'\leq K$, let $S$ be a thread such that its gaps sequence of gaps $\{C^S_i\}_{i\in\mathbb{N}}$ ordered decreasingly in length, verify that $\text{length}(C^S_{i})\leq \gamma^*_i$ for all $i\leq k+1$. Applying inductive hypothesis, since the result is assumed to be true for $k$, we only need to prove that if $\text{length}(C^S_1)<a_{k+1}/K'$, there is no Lipschitz map $F\colon S\rightarrow T_{k+1}$ with $\|F\|_\text{Lip}\leq K'$ and $F(0)=0$ and $F(l_S)=l_{k+1}$. Suppose by contradiction that such a map $F$ exists. Again, we may assume $F$ to be non-decreasing.

Put again $j_0=0$ and $n_{j_0}=1$, and consider the gap $G_{n_{j_0}}^{k+1}$ in $T_{k+1}$ (that is, the biggest gap of $T_{k+1}$). By Lemma \ref{gapjumpsovergap}, there exists $j_1\in\mathbb{N}$ such that the gap $C^S_{j_1}$ in $S$ jumps over $G_{n_{j_0}}^{k+1}$, which has length $\alpha^{k+1}_{n_{j_0}}$. Since $K'\cdot\text{length}(C^S_{k+1})<\alpha^{k+1}_{n_{\Omega_k}}<\alpha^{k+1}_{n_{j_0}}$, by Lemma \ref{onlylonggapscanjump} we have that $j_1\leq k$. Therefore, we can define $n_{(j_0,j_1)}=\min\{n> n_{j_0}\colon ~G^{k+1}_n\nsubseteq D_{(j_0,j_1)}\}$ as before.  

Consider now the gap $G^{k+1}_{n_{(j_0,j_1)}}$ in $T_{k+1}$, and take $j_2$ such that $C^S_{j_2}$ jumps over $G^{k+1}_{n_{(j_0,j_1)}}$. Again, since $K'\cdot\text{length}(C^S_{k+1})<\alpha^{k+1}_{n_{\Omega_k}}<\alpha^{k+1}_{n_{(j_0,j_1)}}$, we obtain that $j_2\leq k$. Moreover, $j_2$ is different from $j_1$. Indeed, if $j_2=j_1$, then $C^S_{j_1}$ jumps over both $G_{n_{j_0}}^{k+1}$ and $G^{k+1}_{n_{(j_0,j_1)}}$. By the choice of $n_{(j_0,j_1)}$, $G^{k+1}_{n_{(j_0,j_1)}}\nsubseteq \mathcal{D}_{K\gamma^*_{j_1}}(G^{k+1}_{n_{j_0}})$, so by Proposition \ref{propertiesdilations}, we have that $K'\cdot\text{length}(C^S_{j_1})>K\gamma^*_{j_1}$, a contradiction.

We can repeat this process $k$-times until we obtain a sequence $\sigma=\{j_i\}_{i=1}^k$ of $k$ different numbers in $\{1,\dots,k\}$ (so $\sigma\in {\Omega_k}$) such that $C_{j_{i}}^S$ jumps over the gap $G^{k+1}_{n_{(j_0,\dots,j_{i-1})}}$ in $T_{k+1}$ for all $1\leq i\leq k$. To finish the proof, consider the gap $G^{k+1}_{n_\sigma}$ in $T_{k+1}$, where $n_\sigma$ is defined as above for $\sigma\in {\Omega_k}$, and take $\tilde{i}$ such that $C^S_{\tilde{i}}$ jumps over $G^{k+1}_{n_\sigma}$. Reasoning as before, by the choice of $\gamma^*_{k+1}$ and using equation \eqref{gammaomegaisthesmallest} we have that $\tilde{i}\leq k$. Then $\tilde{i}=j_{i_0}$ for some $1\leq i_0\leq k$. We have chosen $i_0$ such that $C^S_{j_{i_0}}$ jumps over $G^{k+1}_{n_{(j_0,\dots,j_{i_0-1})}}$ as well. Moreover, $G^{k+1}_{n_\sigma}$ is not contained in $\mathcal{D}_{K\gamma^{*}_{j_{i_0}}}\big( G^{k+1}_{n_{(j_0,\dots,j_{i_0-1})}}\big)$. Therefore, by Proposition \ref{propertiesdilations}, we have that $K'\cdot\text{length}(C^S_{j_{k_0}})>K\gamma^*_{j_{k_0}}$, a contradiction.

The induction is finished and the result follows. 

\end{proof}

Compact subsets of the real line with positive measure that contain no nontrivial intervals have been considered many times before: the well known \emph{fat Cantor sets} are examples of this kind of objects. For our purposes, we need to find threads with these properties and whose gaps are smaller than any given decreasing sequence of positive real numbers. For completeness, we include the construction of such threads in the following subsection.

Let us first finish this subsection by making a simple remark:
\begin{remark}
\label{threadandseparatedsets}
Let $M$ be a complete metric space, let $K\geq 1$, and let $S_1,S_2$ be two closed subsets of $M$ such that $d(S_1,S_2)=\varepsilon>0$. Then, if $T$ is a thread such that its sequence of gaps $\{C^T_k\}_{k\in\mathbb{N}}$ verifies $\text{length}(C^T_k)<\varepsilon/K$ for all $k\in\mathbb{N}$, then there is no $K$-Lipschitz map $F\colon T\rightarrow S_1\cup S_2$ such that $F(0)\in S_1$ and $F(p)\in S_2$ for some $p\in T$.
\end{remark}
\begin{proof}
Consider the point 
\begin{align*}
    P&=\min\{x\in T\colon F(x)\in S_2\}.
\end{align*}
The point $P$ is not $0$ since $F(0)\in S_1$, and for all $x\in [0,P)_T$ we have that $F(x)\in S_1$. However, since every gap in $T$ is smaller than $\varepsilon/K$, there exists $x\in [0,P)_T$ such that $d(x,P)<\varepsilon/K$. Then $d(F(x),F(P))<\varepsilon$, which contradicts the fact that $d(S_1,S_2)=\varepsilon$.
\end{proof}

\subsection{Construction of threads with infinitely many gaps}

Our objective now is to define a collection of closed subsets of the real segment $[0,1]$ containing $0$ and $1$ which, when given a thread metric, will verify the hypothesis of Theorem \ref{nolipschitzmap} for $\varepsilon=1/2$. 
For the rest of this section, fix $\mathbb{Q}\cap (0,1)=(q_n)_{n=1}^\infty$, an ordering of the rational numbers in the interval $(0,1)$.
Consider a decreasing sequence of real numbers $\gamma=(\gamma_i)_{i=1}^\infty$ such that 

\begin{enumerate}[label=(\roman*), ref=(\roman*)]
    \item $\gamma_i>0$ for all $i\in \mathbb{N}$,
    \item $\gamma_i<2^{-(i+1)}$ for all $i\in\mathbb{N}$.
    \item $q_1+\gamma_1<1$.
\end{enumerate}

Put $\Delta=\{\gamma=(\gamma_i)_i\colon \gamma\text{ is decreasing and verifies (i), (ii) and (iii)}\}$ for the rest of the article.

For any given $\gamma\in \Delta$, we define $\{G_i\}_{i\in\mathbb{N}}$ inductively as the following open subintervals of $(0,1)$:
$$    G^\gamma_1=(q_1,q_1+\gamma_1),$$
and for $i\geq 2$:
$$G^\gamma_{i}=(q_{n_i},q_{n_i}+\gamma_i),\text{ where }n_i=\min\bigg\{n\in\mathbb{N}\colon (q_n,q_n+\gamma_i)\subset (0,1)\setminus\bigg(\bigcup_{j<i} G^\gamma_j\bigg) \bigg\}. $$
Note that property $(ii)$ of $\gamma$ guarantees that $n_i$ exists for all $i\in \mathbb{N}$. 

Using this, we define the closed subset $T_\gamma\subset [0,1]$ as

$$T_\gamma=[0,1]\setminus \bigg(\bigcup_{i=1}^\infty G^\gamma_i\bigg) $$
for any $\gamma\in\Delta$. The definition of $\{G^\gamma_i\}_{i\in\mathbb{N}}$ and $T_\gamma$ for every $\gamma\in \Delta$ is fixed for the rest of the article.

\begin{proposition}
\label{propertieslgamma}
Let $\gamma=(\gamma_i)_{i=1}^\infty \in \Delta$, and let $\{G^\gamma_i\}_{i\in\mathbb{N}}$ and $T_\gamma$ be defined as above. Then $T_\gamma$ is a compact subset of $[0,1]$ that verifies:

\begin{enumerate}[label=(\arabic*), ref=(\arabic*)]
    \item The Lebesgue measure of $T_\gamma$ is greater than or equal to $1/2$.
    \item The points $0$ and $1$ belong to $T_{\gamma}$. 
    \item The sequence of gaps of $T_\gamma$ is the sequence $\{G^\gamma_i\}_{i\in\mathbb{N}}$, and $\text{length}(G^\gamma_i)=\gamma_i$ for all $i\in\mathbb{N}$. 
    \item The set $T_\gamma$ does not contain any nontrivial interval. Consequently, if $(x,x+\delta)\cap T_\gamma\neq \emptyset$ for some $x\in [0,1]$ and $\delta>0$, then $(x,x+\delta)$ contains infinitely many of gaps of $T_\gamma$.
\end{enumerate}
\end{proposition}
\begin{proof}
Notice that the Lebesgue measure of $T_\gamma$ is greater than or equal to $1-\sum_{i=1}^\infty 2^{-(i+1)}= 1/2$ for all possible $\gamma$ by property $(ii)$, and the points $0$ and $1$ are not in $G^\gamma_i$ for any $i\in\mathbb{N}$ by construction and property $(iii)$, so $(1)$ and $(2)$ are clear. 

To see $(3)$, we need to prove that $G^\gamma_i$ is a gap in $T^\gamma$ for all $i\in\mathbb{N}$, and that every gap of $T^\gamma$ is one of $G^\gamma_i$ for some $i\in\mathbb{N}$. Consider an interval $G^\gamma_i=(q_{n_i},q_{n_i}+\gamma_i)$. We have directly by construction that $G^\gamma_i\cap T=\emptyset$, so we only need to see that the endpoints of $G^\gamma_i$ are in $T$ to prove that it is a gap of $T$. If one of the endpoints $p_i$ of $G^\gamma_i$ is not in $T$, there must exist $j\in \mathbb{N}$ such that $p_i\in G^\gamma_i$. Since $C^\gamma_j$ is open and $p_i$ is in the closure of $G^\gamma_i$, we have that $G^\gamma_i\cap C^\gamma_j\neq \emptyset$, a contradiction with the choice of $n_i$ and $n_j$. We conclude that $G^\gamma_i$ is a gap of $T_\gamma$.

Next, let $x,y\in T_\gamma$ with $x<y$ and $(x,y)_{T_\gamma}=\emptyset$. For a point $p\in (x,y)$, since $p\notin T_\gamma$, there must exist $i\in\mathbb{N}$ such that $p\in G^\gamma_i$. The interval $G^\gamma_i$ is contained in $(x,y)$ because both $x$ and $y$ belong to $T_\gamma$. Moreover, since the endpoints of $G^\gamma_i$ belong to $T_\gamma$, we necessarily have that $G^\gamma_i=(x,y)$, and we are done with $(3)$. 

Finally, the set $T_\gamma$ is nowhere dense, since it contains no intervals. Indeed, suppose there is an interval $(x,x+\delta)\subset T_\gamma$ for some $x\in [0,1]$ and $\delta>0$ with $x+\delta<1$. The subinterval $(x,x+\delta/2)$ contains a rational number $q_{n_0}$. Since $(\gamma_i)_{i=1}^\infty$ is decreasing and converging to $0$, there must exist $i_0$ such that $\gamma_{i}<\delta/2$ for all $i\geq i_0$. Then, for all $i\geq i_0$, the natural number $n_0$ verifies that $(q_{n_0},q_{n_0}+\gamma_i)\subset T_\gamma$, and in particular
$$(q_{n_0},q_{n_0}+\gamma_i)\subset (0,1)\setminus\bigg(\bigcup_{j<i}G_j^\gamma\bigg). $$
Therefore, there must exist $i_1\geq i_0$ such that $n_0= \min\bigg\{n\in\mathbb{N}\colon (q_n,q_n+\gamma_{i_1})\subset(0,l)\setminus\bigg(\bigcup_{j<i_1}G_j^\gamma\bigg)\bigg\}$, which implies that $G^\gamma_{i_1}=(q_{n_0},q_{n_0}+\gamma_{i_1})$; a contradiction with the assumption that $(q_{n_0},q_{n_0}+\delta/2)\subset T_\gamma$.
\end{proof}

Now, given any $\gamma\in\Delta$ and any $0<a\leq 1$, we may assign the metric $d_{1,a}$ as defined at the beginning of this section to the set $T_\gamma$, such that $(T_\gamma,d_{1,a})$ is a thread. We will denote by $T_\gamma(1,a)$ the thread of length $1$ and width $a$ formed by endowing the subset $T_\gamma$ as defined above for $\gamma\in \Delta$ with the metric $d_{1,a}$.

With Proposition \ref{propertieslgamma} we have that any countable family of these threads verifies the hypothesis of Theorem \ref{nolipschitzmap}. In fact, any countable family of subthreads with measure uniformly bounded from below also verifies the hypothesis of Theorem \ref{nolipschitzmap}. Moreover, given such a countable family of threads $\{T_n\}_{n\in\mathbb{N}}$ and $K\geq 1$, for the sequence $\gamma^*=\{\gamma^*_k\}_{k\in\mathbb{N}}$ obtained by Theorem \ref{nolipschitzmap} we can always find another sequence $\gamma^0=\{\gamma^0_k\}_{k\in\mathbb{N}}\in \Delta$ such that $\gamma^0_k\leq\gamma^*_k$ for all $k\in\mathbb{N}$. Hence, there exists a thread of the form $T_{\gamma^0}(1,a)$ that cannot be mapped with a $K$-Lipschitz function preserving the extreme points onto any $T_n$ for any $n\in\mathbb{N}$, provided $a<\text{width}(T_n)/K$.

Observe that thanks to the properties of the generic sets $T_\gamma$, we can obtain the following fact about the thread $T_\gamma(1,a)$:
\begin{proposition}
\label{totallyseparated}
Let $\gamma\in \Delta$, and let $T_\gamma(1,a)$ be the thread of length $1$ and width $a$ associated with $\gamma$. Then $T_\gamma(1,a)$ is totally separated, i.e.: If $p$ and $q$ are two different points in $T_\gamma(1,a)$, there exist two disjoint open and closed subsets $S_1,S_2\subset T_\gamma(1,a)$ such that $p\in S_1$, $q\in S_2$, and $T_\gamma(1,a)=S_1\cup S_2$. As a consequence, for any point $p\in T_\gamma(1,a)$ and any $\varepsilon>0$, there exists an open and closed subset $S$ of diameter less than $\varepsilon$ such that $p\in S$. 
\end{proposition}
\begin{proof}
Put $T=T_\gamma(1,a)$. Let $p$ and $q$ be two different points in $T$. Suppose without loss of generality that $p<q$. If the interval $(p,q)_T$ is empty, then the result follows considering $S_1=[0,p]_T$ and $S_2=[q,1]_T$. Otherwise, if $(p,q)_T$ is nonempty, by property (4) in Proposition \ref{propertieslgamma} there exists a gap $(x,y)$ in $T$ such that $p<x$ and $y<q$. Put now $S_1=[0,x]_T$ and $S_2=[y,1]_T$ and the result is proven.

The second statement follows immediately from the first and from the linear structure of threads.
\end{proof}
When dealing with Lipschitz functions between threads, the image of the extreme points of a thread plays an important role for some technical arguments (as can be seen in the previous subsection). This is the motivation for introducing the next result.

\begin{proposition}
\label{changingextremes}
Let $T$ and $S$ be two threads of length $l_T$ and $l_S$, and width $a_T$ and $a_S$ respectively. Let $K\geq 1$. Suppose $T$ is totally separated. Consider $S_1$ and $S_2$ open subsets of $S$, and take $D_1$ and $D_2$ dense subsets of $S_1$ and $S_2$ respectively. 

Let $F\colon T\rightarrow S$ be a $K$-Lipschitz function such that $F(0)=A$ and $F(l_S)=B$ with $A\in S_1$ and $B\in S_2$. Then for every $\varepsilon>0$, there exists a pair of points $P,Q\in T$ with $P<Q$, a pair of points $\widehat{A}\in D_1$ and $\widehat{B}\in D_2$, and a $(K+\varepsilon)$-Lipschitz function $\widehat{F}\colon [P,Q]_T\rightarrow S$ such that $\widehat{F}(P)=\widehat{A}$ and $\widehat{F}(Q)=\widehat{B}$.
\end{proposition}
\begin{proof}
Since $S_1$ is open in $S$ and $F(0)\in S_1$, there exists a positive number $r>0$ such that for any point $p\in [0,r]_T$ we have $F(p)\in S_1$. The thread $T$ is totally separated, so we can find $P\in [0,r]_T$ and $\delta_P>0$ such that $d\big([0,P]_T,T\setminus\big([0,P]_T\big)\big)=\delta_P$. Similarly, we may find $Q\in T$ with $P<Q$ and $\delta_Q>0$ such that $F(Q)\in S_2$ and $d\big([Q,l_T]_T,T\setminus\big([Q,l_T]_T\big)\big)=\delta_Q$.

By density, we can find $\widehat{A}\in S_1$ and $\widehat{B}\in S_2$ such that $d\big(F(P),\widehat{A}\big)<2^{-1}\varepsilon\cdot \delta_P$ and $d\big(F(Q),\widehat{B}\big)<2^{-1}\varepsilon\cdot \delta_Q$. Define now $\widehat{F}\colon [P,Q]_T\rightarrow S$ so that 
$$ 
\widehat{F}(x)=
\begin{cases}
\widehat{A},~&\text{ if } x=P,\\
F(x),~&\text{ if } x\in (P,Q)_T,\\
\widehat{B},~ &\text{ if } x=Q.
\end{cases}
$$
It is now routine to check that $\widehat{F}$ is $(K+\varepsilon)$-Lipschitz.
\end{proof}

\section{Construction of the building blocks: Threading metric spaces}

We now want to use the threads $T_\gamma(1,a)$ we defined for $\gamma\in\Delta$ and $0<a\leq 1$ to construct non-separable complete metric spaces that will act as building blocks of the final metric space. To do this, we first formalize the notion of \emph{attachment} of metric spaces, which will allow us to ``glue" metric spaces in a convenient way. This concept has been used in many contexts in the literature, but we choose to include a definition tailored to our necessities.

\begin{definition}
Let $(M,d_M)$ be a complete metric space. Let $\mathcal{N}=\{(N_\gamma,d_\gamma)\}_{\gamma\in \Gamma}$ be a collection of pairwise disjoint and disjoint with $M$ complete metric spaces, and let $\mathcal{S}=\{S_\gamma\}_{\gamma\in \Gamma}$ be a collection of sets such that: for each $\gamma\in \Gamma$ the set $S_\gamma$ is a compact subset of $N_\gamma$, and there exists an isometry $\Phi_\gamma\colon S_\gamma\rightarrow M$ onto a subset of $M$. The \emph{attachment of $M$ with $\mathcal{N}$ by $\mathcal{S}$} is the pair $(M(\mathcal{N},\mathcal{S}),d_{\mathcal{N},\mathcal{S}})$, where 
$$M(\mathcal{N},\mathcal{S})=M\cup\bigg(\bigcup_{\gamma\in\Gamma} N_\gamma\setminus S_\gamma\bigg), $$
and 
\begin{align*}
   d_{\mathcal{N},\mathcal{S}}\colon &  M(\mathcal{N},\mathcal{S})\times M(\mathcal{N},\mathcal{S})\longrightarrow \mathbb{R}^+\\
\end{align*} 
is a map defined by 
$$ 
\footnotesize
d_{\mathcal{N},\mathcal{S}}(p,q)=
\begin{cases}
d_M(p,q),&\text{ if } p,q\in M,\\
d_\gamma(p,q), &\text{ if }p,q\in N_\gamma\setminus S_\gamma\text{ for some }\gamma\in \Gamma,\\
\min_{x\in S_\gamma}\{d_\gamma(p,x)+d_M(\Phi_\gamma(x),q)\}, &\text{ if }p\in N_\gamma\setminus S_\gamma\text{ for some }\gamma\in \Gamma,\text{ and }q\in M,\\
H_{\gamma_1,\gamma_2}(p,q), &\text{ if } p\in N_{\gamma_1}\setminus S_{\gamma_1},~ q\in N_{\gamma_2}\setminus S_{\gamma_2}\text{ for }\gamma_1\neq \gamma_2\in \Gamma,
\end{cases}
$$
where $H_{\gamma,\eta}\colon N_\gamma\times N_\eta\rightarrow \mathbb{R}^+$ is the map defined by 
$$ H_{\gamma,\eta}(p,q)=\min\{d_\gamma(p,x)+d_M(\Phi_\gamma(x),\Phi_\eta(y))+d_\eta(y,q)\colon x\in S_\gamma,~ y\in S_\eta\}.$$
\end{definition}
Notice that both minima used in the definition of $d_{\mathcal{N},\mathcal{S}}$ are well defined by compactness of the sets $S_\gamma$ for each $\gamma\in \Gamma$. Moreover, it is straightforward to check that the map $d_{\mathcal{N},\mathcal{S}}$ defines a complete metric in $M(\mathcal{N},\mathcal{S})$. 

It is also clear from the definition that the metric space $M(\mathcal{N},\mathcal{S})$ contains $M$ isometrically, as well as an isometric copy of $N_\gamma$ for each $\gamma\in\Gamma$. We may write $M\subset M(\mathcal{N},\mathcal{S})$ and $N_\gamma\subset M(\mathcal{N},\mathcal{S})$ by virtue of this fact.

With this concept, we can now define the aforementioned building blocks of the main metric space we seek to construct:

\begin{definition}
Consider a metric space $M=\{A,B\}$ formed by two points at a distance $0<a\leq 1$. Let $\mathcal{N}_a=\{T_{\gamma}(1,a)\}_{ \gamma\in\Delta}$, where $\Delta$ is the set of sequences defined in Section 1, and $T_\gamma(1,a)$ is the thread associated with $\gamma$ of width $a$. We may consider $T_\gamma(1,a)$ and $T_\eta(1,a)$ to be disjoint for $\gamma\neq \eta$. For each $T_{\gamma}(1,a)$, put $S_\gamma=\{0_\gamma,1_\gamma\}$, the set of the two extreme points of $T_{\gamma}(1,a)$. We let $\mathcal{S}_a=\{S_\gamma\}_{\gamma\in\Delta}$, and
$\Phi_\gamma:S_\gamma\to M$ as $\Phi_\gamma(0_\gamma)=A,\;\Phi_\gamma(1_\gamma)=B$.

We define the \emph{threading space} $\text{Th}(A,B)$ to be the attachment of $M$ with $\mathcal{N}_a$ by $\mathcal{S}_a$. We say that $\text{Th}(A,B)$ is \emph{anchored} at $A$ and $B$, and these two points are called the \emph{anchors} of $\text{Th}(A,B)$. If a threading space $\text{Th}(A,B)$ is fixed and there is no room for ambiguity, we write $T_\gamma(1,a)\subset \text{Th}(A,B)$ to denote the isometric copy of the thread $T_\gamma(1,a)$ contained in the threading space $\text{Th}(A,B)$.
\end{definition}
Note that in a threading space $\text{Th}(A,B)$ we have that $T_\gamma(1,a)\cap T_\eta(1,a)=\{A,B\}$ for any two different $\gamma,\eta\in\Delta$. Moreover, if $p,q\in \text{Th}(A,B)$ belong to different threads, the distance $d(p,q)$ is computed according to the definition of attachment, which in this case results in 
$$d(p,q)=\min\{d(p,A)+d(A,q),d(p,B)+d(B,q)\}. $$
See Figure \ref{Threadingspacevisualrepr} for a representation of a subset of a threading space. 
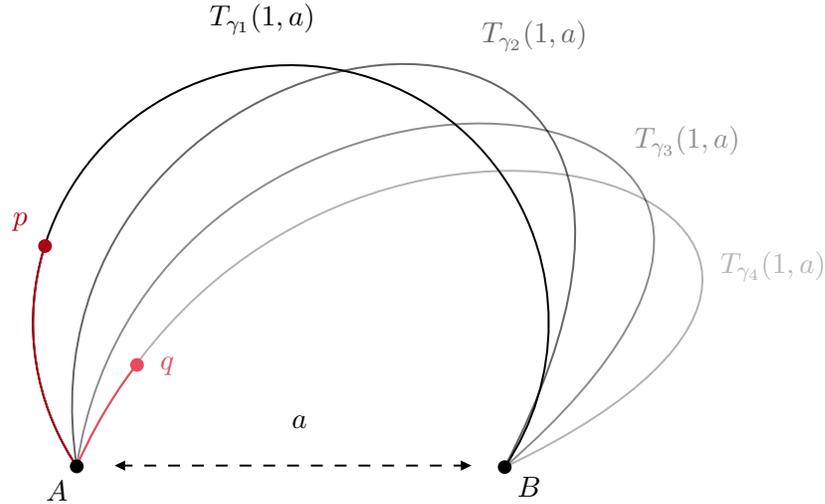
\begin{figure}
    \centering

\tikzset{every picture/.style={line width=0.75pt}} 

\begin{tikzpicture}[x=0.75pt,y=0.75pt,yscale=-1,xscale=1]

\draw  [draw opacity=0] (193,273.89) .. controls (193,273.89) and (193,273.89) .. (193,273.89) .. controls (193,273.89) and (193,273.89) .. (193,273.89) .. controls (151.79,215.03) and (166.05,133.94) .. (224.87,92.75) .. controls (283.68,51.57) and (364.76,65.9) .. (405.98,124.76) .. controls (437.76,170.15) and (436.55,228.76) .. (407.36,272.09) -- (299.49,199.33) -- cycle ; \draw   (193,273.89) .. controls (193,273.89) and (193,273.89) .. (193,273.89) .. controls (193,273.89) and (193,273.89) .. (193,273.89) .. controls (151.79,215.03) and (166.05,133.94) .. (224.87,92.75) .. controls (283.68,51.57) and (364.76,65.9) .. (405.98,124.76) .. controls (437.76,170.15) and (436.55,228.76) .. (407.36,272.09) ;
\draw  [draw opacity=0] (193,273.89) .. controls (193,273.89) and (193,273.89) .. (193,273.89) .. controls (193,273.89) and (193,273.89) .. (193,273.89) .. controls (169.12,239.8) and (163.87,198.24) .. (175.1,161.3) -- (299.49,199.33) -- cycle ; \draw  [color={rgb, 255:red, 169; green, 3; blue, 23 }  ,draw opacity=1 ] (193,273.89) .. controls (193,273.89) and (193,273.89) .. (193,273.89) .. controls (193,273.89) and (193,273.89) .. (193,273.89) .. controls (169.12,239.8) and (163.87,198.24) .. (175.1,161.3) ;
\draw [color={rgb, 255:red, 0; green, 0; blue, 0 }  ,draw opacity=0.3 ]   (191.52,271.74) .. controls (306.67,15.29) and (685.07,139.85) .. (407.36,272.09) ;
\draw [color={rgb, 255:red, 231; green, 75; blue, 93 }  ,draw opacity=1 ]   (191.52,271.74) .. controls (203.47,243.05) and (221.87,219.85) .. (221.92,220.54) ;
\draw  [fill={rgb, 255:red, 0; green, 0; blue, 0 }  ,fill opacity=1 ] (188.56,271.74) .. controls (188.56,270.1) and (189.89,268.78) .. (191.52,268.78) .. controls (193.16,268.78) and (194.48,270.1) .. (194.48,271.74) .. controls (194.48,273.38) and (193.16,274.7) .. (191.52,274.7) .. controls (189.89,274.7) and (188.56,273.38) .. (188.56,271.74) -- cycle ;
\draw  [fill={rgb, 255:red, 0; green, 0; blue, 0 }  ,fill opacity=1 ] (404.4,272.09) .. controls (404.4,270.45) and (405.73,269.13) .. (407.36,269.13) .. controls (409,269.13) and (410.33,270.45) .. (410.33,272.09) .. controls (410.33,273.73) and (409,275.05) .. (407.36,275.05) .. controls (405.73,275.05) and (404.4,273.73) .. (404.4,272.09) -- cycle ;
\draw  [dash pattern={on 4.5pt off 4.5pt}]  (213.6,271.33) -- (387.6,271.33) ;
\draw [shift={(390.6,271.33)}, rotate = 180] [fill={rgb, 255:red, 0; green, 0; blue, 0 }  ][line width=0.08]  [draw opacity=0] (5.36,-2.57) -- (0,0) -- (5.36,2.57) -- cycle    ;
\draw [shift={(210.6,271.33)}, rotate = 0] [fill={rgb, 255:red, 0; green, 0; blue, 0 }  ][line width=0.08]  [draw opacity=0] (5.36,-2.57) -- (0,0) -- (5.36,2.57) -- cycle    ;
\draw [color={rgb, 255:red, 0; green, 0; blue, 0 }  ,draw opacity=0.6 ]   (191.52,271.74) .. controls (154.67,31.29) and (562.95,-27.91) .. (407.36,272.09) ;
\draw [color={rgb, 255:red, 0; green, 0; blue, 0 }  ,draw opacity=0.45 ]   (191.52,271.74) .. controls (229.07,22.49) and (650.95,60.09) .. (407.36,272.09) ;
\draw  [color={rgb, 255:red, 169; green, 3; blue, 23 }  ,draw opacity=1 ][fill={rgb, 255:red, 169; green, 3; blue, 23 }  ,fill opacity=1 ] (172.56,160.54) .. controls (172.56,158.9) and (173.89,157.58) .. (175.52,157.58) .. controls (177.16,157.58) and (178.48,158.9) .. (178.48,160.54) .. controls (178.48,162.18) and (177.16,163.5) .. (175.52,163.5) .. controls (173.89,163.5) and (172.56,162.18) .. (172.56,160.54) -- cycle ;
\draw  [color={rgb, 255:red, 231; green, 75; blue, 93 }  ,draw opacity=1 ][fill={rgb, 255:red, 231; green, 75; blue, 93 }  ,fill opacity=1 ] (218.96,220.54) .. controls (218.96,218.9) and (220.29,217.58) .. (221.92,217.58) .. controls (223.56,217.58) and (224.88,218.9) .. (224.88,220.54) .. controls (224.88,222.18) and (223.56,223.5) .. (221.92,223.5) .. controls (220.29,223.5) and (218.96,222.18) .. (218.96,220.54) -- cycle ;
\draw [draw opacity=0]   (102,168) -- (575.92,170.8) ;

\draw (174.59,277.31) node [anchor=north west][inner sep=0.75pt]    {$A$};
\draw (412.33,275.49) node [anchor=north west][inner sep=0.75pt]    {$B$};
\draw (298.6,244.73) node [anchor=north west][inner sep=0.75pt]    {$a$};
\draw (157.79,142.08) node [anchor=north west][inner sep=0.75pt]  [color={rgb, 255:red, 169; green, 3; blue, 23 }  ,opacity=1 ]  {$p$};
\draw (232.19,214.88) node [anchor=north west][inner sep=0.75pt]  [color={rgb, 255:red, 231; green, 75; blue, 93 }  ,opacity=1 ]  {$q$};
\draw (256.99,36.72) node [anchor=north west][inner sep=0.75pt]    {$T_{\gamma _{1}}( 1,a)$};
\draw (394.39,45.32) node [anchor=north west][inner sep=0.75pt]  [color={rgb, 255:red, 0; green, 0; blue, 0 }  ,opacity=0.6 ]  {$T_{\gamma _{2}}( 1,a)$};
\draw (471.39,98.52) node [anchor=north west][inner sep=0.75pt]  [color={rgb, 255:red, 0; green, 0; blue, 0 }  ,opacity=0.45 ]  {$T_{\gamma _{3}}( 1,a)$};
\draw (514.59,161.52) node [anchor=north west][inner sep=0.75pt]  [color={rgb, 255:red, 0; green, 0; blue, 0 }  ,opacity=0.3 ]  {$T_{\gamma _{4}}( 1,a)$};

\end{tikzpicture}

    \caption{Subset of $\text{Th}(A,B)$. The distance from $p$ to $q$ is $d(p,A)+d(A,q)$.}
    \label{Threadingspacevisualrepr}
\end{figure}

By definition, in a metric space $M(\mathcal{N},\mathcal{S})$ formed by attachment (and in threading spaces in particular), given a point $p\in N_\gamma $ for some $\gamma\in \Gamma$ and a point $x_1\in M$, there exists a point $s_1\in S_\gamma$ such that $d(p,x_1)=d(p,s_1)+d(s_1,x_1)$. However, it is possible that for a different point $x_2\in M$, the point $s_2\in S_\gamma$ such that the identity $d(p,x_2)=d(p,s_2)+d(s_2,x_2)$ holds, is different from $s_1$. Points in $N_\gamma$ that use always the same point in $S_\gamma$ to compute their distance to the rest of the space are especially relevant to our discussion. 

In general, given a metric space $M$ and a closed subset $N$, we say that a point $p\in N$ is \textit{bound to }$s\in N$ in $N$ if for every $x\in M\setminus N$ we have $d(p,x)=d(p,s)+d(s,x)$.

For example, in a threading space $\text{Th}(A,B)$, it is not hard to check that a point $p$ in a thread $T_\gamma(1,a)$ is bound to $A$ in $T_\gamma(1,a)$ if and only if $d(p,A)<d(p,B)$ and $d(p,A)\leq \frac{1-a}{2}$ (the analogous result for $B$ holds as well).

In the final section of the article we will deal with Lipschitz functions defined on a single thread and with image in metric spaces formed by attachment. We finish this section by defining two simple concepts and proving a result that will help us simplify this type of maps.

Now, if $T$ is a thread, $M$ is a metric space, $N$ is a closed subset of $M$, and $F\colon T\rightarrow M$ is a Lipschitz function, we say that an extended interval $I$ is \textit{maximal} with respect to $F$ and $N$ if $F(I)\subset N$ and every extended interval $J=[a',b']_T$ in $T$ that contains $I$ and such that $F(J)\subset N$ is equal to $I$. We have the following straightforward result:

\begin{proposition}
\label{maximalwithentryeqexit}
Let $T$ be a thread of length $l$ and width $a$. Let $M$ be a metric space, let $N$ be a closed subset of $M$, and let $F\colon T\rightarrow M$ be a Lipschitz function. If an extended interval $I$ in $T$ with extremes $a,b\in T$ is maximal with respect to $F$ and $N$, and there exists $s\in N$ such that both $F(a)$ and $F(b)$ are bound to $s$ in $N$, then the function $\widehat{F}\colon T\rightarrow M$ defined by:
$$
\widehat{F}(x)=
\begin{cases}
s,&\text{ if } x\in I,\\
F(x), &\text{ if }x\in T\setminus I
\end{cases}
$$
is Lipschitz with $\|\widehat{F}\|_\text{Lip}=\|F\|_\text{Lip}$.
\end{proposition}
\begin{proof}
Put $K=\|F\|_\text{Lip}$. We will start by proving that
\begin{equation}
\label{widehatFverifiescondition}
d\big(\widehat{F}(0),\widehat{F}(l)\big)\leq K\cdot a.    
\end{equation}
If both $0$ and $l$ belong to the extended interval $I$ then it follows trivially. Similarly, if $0$ and $l$ belong to $T\setminus I$ then the inequality follows since $F$ is $K$-Lipschitz. Hence, suppose first that $0\in I$ and $l\in T\setminus I$. Then, we have necessarily that $a= 0$, and so $F(0)$ is bound to $s$ in $N$. This implies that 
\begin{align*}
    d\big(\widehat{F}(0),\widehat{F}(l)\big)&=d\big(s,F(l)\big)\leq d\big(F(0),s\big)+d\big(s,F(l)\big)\\
    &=d\big(F(0),F(l)\big)\leq K\cdot a
\end{align*}
A similar argument shows that if $0\in T\setminus I$ and $l\in I$ the same inequality holds. Hence we conclude that equation \eqref{widehatFverifiescondition} is verified.

Next, we prove that for every $x,y\in T$ with $x<y$ we have
\begin{equation}
\label{widehatFverifiescondition2}
d(\widehat{F}(x),\widehat{F}(y))\leq K(y-x).
\end{equation}

As before, we may only check this holds for $x,y\in T$ with $x<y$ such that $x\in T\setminus I$ and $y\in I$. By maximality of $I$, there exists $t\in T$ with $x\leq t<y$ such that $F(t)\notin N$. Additionally, since $t$ is not in $I$ but $y$ does belong to the extended interval, one of the extremes $a$ or $b$ of $I$ belongs to $(t,y]_T$. We may suppose without loss of generality that $x\leq t<a\leq y$. Notice that $d\big(F(t),s\big)\leq d\big(F(t),F(a)\big)$ because $a$ is bound to $s$ in $N$. Then we have:
\begin{align*}
    d\big(\widehat{F}(x),\widehat{F}(y)\big)&\leq d\big(F(x),F(t)\big)+d\big(F(t),s\big)\\
    &\leq d\big(F(x),F(t)\big)+d\big(F(t),F(a)\big)+d\big(F(a),F(y)\big)\\
    &\leq  K\big((t-x)+(a-t)+(y-a)\big)=K(y-x).
\end{align*}
This proves that equation \eqref{widehatFverifiescondition2} holds as well.

Using both equations \eqref{widehatFverifiescondition} and \eqref{widehatFverifiescondition2} we can apply Proposition \ref{domainthreadisinterval} to obtain that $\|\widehat{F}\|_\text{Lip}\leq K$.
\end{proof}

The process used to construct the skein metric space which fails to have any non-trivial separable Lipschitz retracts is to keep attaching threading spaces inductively such that any two distinct points of the skein act as the anchors to a threading space contained in the skein\footnote{``\emph{skein}: a length of yarn or thread collected together into the shape of a loose ring'' (\citefield{Skein}{booktitle} \citeyear{Skein}).}. As we are going to see in the final section, this construction presents its own technical difficulties. However, the main results of the first two sections will be very useful in this regard.

\section{Construction of the skein metric spaces}

The final metric space will be constructed using transfinite induction. Let us discuss this process in general for limit ordinal numbers: 

Let $\kappa$ be a limit ordinal number. Suppose that $\{(M_\alpha,d_\alpha)\}_{\alpha<\kappa}$ is a transfinite sequence of metric spaces that are \emph{increasing}, in the sense that $M_\alpha\subset M_\beta$ if $\alpha<\beta$ and the restriction of $d_\beta$ to $M_\alpha$ results in the metric $d_\alpha$. Then we may define the metric space $(M_\kappa,d_\kappa)$ where $M_\kappa=\bigcup_{\alpha<\kappa}M_\alpha$, and $d_\kappa$ is defined for any $p,q\in M_\kappa$ as $d_\kappa(p,q)=d_\alpha(p,q)$ where $\alpha<\kappa$ is the least ordinal number such that $p,q\in M_\alpha$. It is straightforward to check that the metric $d_\kappa$ is well defined and $(M_\kappa,d_\kappa)$ is indeed a metric space.

We will call $(M_\kappa,d_\kappa)$ the metric space \emph{generated by $\{(M_\alpha,d_\alpha)\}_{\alpha<\kappa}$}, and as usual we may omit the mention of the metric $d_\kappa$ when referring to it if there is no room for ambiguity. If $\kappa$ is an ordinal with uncountable cofinality (i.e., the supremum of any countable sequence of ordinals $(\alpha_n)_n$ such that $\alpha_n<\kappa$ for all $n\in\mathbb{N}$ is strictly smaller than $\kappa$), then the metric space $M_\kappa$ generated by $\{M_\alpha\}_{\alpha<\kappa}$ is complete, provided each $M_\alpha$ is complete for every $\alpha<\kappa$. To see this, consider any Cauchy sequence $(p_n)_n$ in $M_\kappa$. Each $p_n$ belongs to $M_{\alpha_n}$ for some ordinal $\alpha_n<\kappa$. Since $\kappa$ has uncountable cofinality, the supremum $\alpha^*=\sup_n(\alpha_n)$ is strictly smaller than $\kappa$. Hence the sequence $(p_n)_n$ belongs to the complete metric space $M_{\alpha^*}$, and therefore it is convergent in $M_{\alpha^*}$ to a point $p^*$. The point $p^*$ belongs to $M_\kappa$, and clearly $(p_n)_n$ converges to $p^*$ in $M_\kappa$ as well. 

\subsection{Construction of the skein metric spaces}

We are going to construct by transfinite induction an increasing class of complete metric spaces $\{\text{Sk}(\beta)\}_{\beta}$ for every ordinal $\beta$, called the \emph{$\beta$-skein metric spaces}. The complete metric space failing to have any non-trivial separable Lipschitz retract is the $\omega_1$-skein space $\text{Sk}(\omega_1)$ (or, more generally, any $\beta$-skein space such that the cofinality of $\beta$ is uncountable).

Consider at the first step the $0$-skein metric space $M_0=\{A,B\}$ formed by two points at distance $1/2$, and put $G_0=\{A,B\}$. Suppose we have defined incresingly the $\alpha$-skein spaces $\{\text{Sk}(\alpha)\}_{\alpha<\beta}$ up to an ordinal $\beta$. If $\beta$ is a limit ordinal, simply define $\text{Sk}(\beta)$ as the completion of generated metric space $\bigcup_{\alpha<\beta}\text{Sk}(\alpha)$ in the way described above, which contains isometrically the previous skein spaces $\text{Sk}(\alpha)$ for all $\alpha<\beta$. Notice that if $\beta$ has uncountable cofinality, it is not necessary to take the completion.

Suppose now that $\beta=\lambda+1$ for an ordinal $\lambda$. For every $p$ in the skein $\text{Sk}(\lambda)$ and every $q\in G_\lambda=\text{Sk}(\lambda)\setminus \bigg(\bigcup_{\alpha<\lambda}\text{Sk}(\alpha)\bigg)$ with $p\neq q$ and $d(p,q)\leq 1/2$, we may consider the threading space $\text{Th}(p,q)$ as defined in section 3. Take now the family of complete metric spaces $\mathcal{N}_\lambda=\{\text{Th}(p,q)\}_{\{p,q\}\in \Gamma_\lambda}$, where 
$$\Gamma_\lambda=\big\{\{p,q\}\subset \text{Sk}(\lambda)\colon p\in \text{Sk}(\lambda),~ q\in G_\lambda,~0<d(p,q)\leq 1/2\big\},$$
which we may take to be pairwise disjoint and disjoint with $\text{Sk}(\lambda)$. For any $\{p,q\}\in \Gamma_\lambda$, we have by definition of the threading space $\text{Th}(p,q)$ that there is an isometry $\Phi_{\{p,q\}}$ from the set of anchor points $\text{An}_{\{p,q\}}$ of $\text{Th}(p,q)$ onto the set $\{p,q\}$ in $\text{Sk}(\lambda)$. Therefore, considering $\mathcal{S}_\lambda=\{\text{An}_{\{p,q\}}\}_{\{p,q\}\in\Gamma_\lambda}$ we can define $\text{Sk}(\beta)$ as the attachment of $\text{Sk}(\lambda)$ with $\mathcal{N}_\lambda$ by $\mathcal{S}_\lambda$. The resulting metric space $\text{Sk}(\beta)$ is the $\beta$-skein and it is a complete metric space containing isometrically the previous skein space $\text{Sk}(\lambda)$. The induction process is finished and we have defined the $\beta$-skein metric space for every ordinal number $\beta$.

Intuitively, we may describe the previous process in the following way: If $\beta$ is a limit ordinal, then the $\beta$-skein space is the completion (if necessary) of the union of all previous skein spaces. If $\beta$ is the successor of an ordinal $\lambda$, then the $\beta$-skein is formed by attaching a threading space at every pair of points closer than $1/2$ and such that at least one of them was newly introduced at the previous step $\lambda$. 

For a subset $S$ of a skein space $\text{Sk}(\beta)$, we may define its \emph{(Skein) order}, written $\text{ord}(S)$, as the least ordinal $\alpha\leq\beta$ such that $S\subset \text{Sk}(\alpha)$. For a point $p\in \text{Sk}(\beta)$, we write $\text{ord}(p)=\text{ord}(\{p\})$. For any ordinal $\beta$, the \emph{(skein) generation of order $\beta$} is the set $G_\beta=\text{Sk}(\beta)\setminus\bigg(\bigcup_{\alpha<\beta}\text{Sk}(\alpha)\bigg)$. 

\begin{figure}
    \centering
    \includegraphics[width=\textwidth]{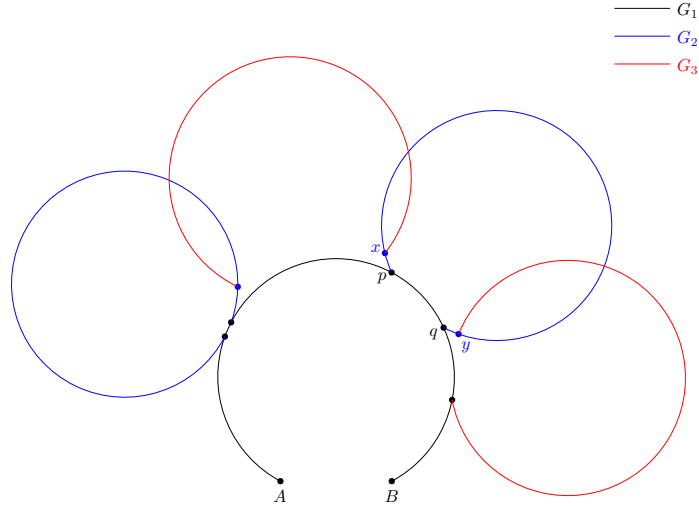}
    \caption{Subset of the skein space $\text{Sk}(3)$. The points $x$ and $y$ are bound to $p$ and $q$ respectively.}    
    \label{SkeinSpace3}
\end{figure}

Figure \ref{SkeinSpace3} is a conceptual representation of a subset of the skein $\text{Sk}(3)$, which contains 3 different generations (the gaps in the threads have been ignored for the sake of clarity). The distance between the points $x$ and $y$ in the figure are computed by $d(x,y)=d(x,p)+d(p,q)+d(q,y)$. 

Crucially, if $\beta$ has uncountable cofinality, the corresponding generation $G_\beta$ is empty, and every point in the $\beta$-skein $\text{Sk}(\beta)$ belongs to a previous generation. This means that, in such a skein $\text{Sk}(\beta)$, every pair of points $p$ and $q$ such that $d(p,q)\leq 1/2$ belong to a set $\Gamma_\alpha$ where $\alpha$ is strictly smaller than $\beta$, and thus an isometric copy of the threading space $\text{Th}(p,q)$ is contained in $\text{Sk}(\beta)$. Moreover, in this case, the order of any separable subset of the skein space $\text{Sk}(\beta)$ is strictly smaller than $\beta$.

We turn our attention now to the specific skein space $\text{Sk}(\omega_1)$. However, as we mentioned, the results of the rest of the article concerning the $\omega_1$-skein can also be written for any $\beta$-skein where $\beta$ is an ordinal with uncountable cofinality.

Notice that for any two different pairs of different points $(p_1,q_1),(p_2,q_2)\in \text{Sk}(\omega_1)\times \text{Sk}(\omega_1)$ such that $d(p_1,q_1)=d(p_2,q_2)\leq 1/2$, the threading spaces $\text{Th}(p_1,q_1)$ and $\text{Th}(p_2,q_2)$ are contained in $\text{Sk}(\omega_1)$ and are isometric. Moreover, for any $\gamma\in \Delta$, each of these two threading spaces contains an isometric copy of the thread $T_\gamma(1,a)$, where $a=d(p_1,q_1)$. To differentiate the different copies of the same thread in $M$ that arise due to this fact, we will denote by $T_\gamma(p,q)$ the thread $T_\gamma(1,d(p,q))$ contained in the threading space $\text{Th}(p,q)\subset \text{Sk}(\omega_1)$. 

Finally, note also that for a given successor ordinal number $\beta+1$ and any $(p,q)\in \Gamma_{\beta}$ and $\gamma\in \Delta$, it holds that $T_{\gamma}(p,q)\setminus \{p,q\}$ is open in the skein $\text{Sk}(\beta+1)$. Hence, we conclude that any open subset of a thread $T_{\gamma}(p,q)\subset \text{Sk}(\beta+1)$ with $(p,q)\in \Gamma_{\beta}$ and $\gamma\in \Delta$ which does not contain the extreme points $\{p,q\}$ is also open in $\text{Sk}(\beta+1)$.

The skein space $\text{Sk}(\omega_1)$ contains separable subsets with different structures, all of which fail to be Lipschitz retracts of $\text{Sk}(\omega_1)$. We are going to prove some results that let us reduce the kind of separable subsets we have to consider to a smaller class. In particular, first we are going to show that it is enough to prove that separable subsets without isolated points are not Lipschitz retracts. Secondly, we will introduce some concepts and prove some results to deal with points in limit ordinal generations. We structure these two topics in two different subsections:

\subsection{First reduction: subspaces with isolated points}
This first reduction is relatively straightforward to see. It is based on two quick observations about the skein space $\text{Sk}(\omega_1)$ and about threads with small gaps. The first observation is a general fact about threads which we have already stated and proven in remark \ref{threadandseparatedsets}. The second one we present in the following simple lemma:
\begin{lemma}
\label{chaininmainspace}
Let $p,q\in \text{Sk}(\omega_1)$ be two different points. There exists a finite sequence $\{x_k\}_{k=0}^n\subset\text{Sk}(\omega_1)$ with $x_0=p$ and $x_n=q$ such that $d(x_k,x_{k+1})\leq 1/2$ for all $0\leq k\leq n-1$.
\end{lemma}
\begin{proof}
We prove the result by transfinite induction on $\text{ord}(\{p,q\})<\omega_1$. If $\text{ord}(\{p,q\})=0$, then $\{p,q\}=\{A,B\}$ and the result follows directly.
Suppose $\text{ord}(\{p,q\})=\beta<\omega_1$, and suppose the result is true for any set of two points with order $\alpha<\beta$. Consider $\beta_p=\text{ord}(p)$. If $\beta_p$ is a limit ordinal, then $p$ is the limit of a sequence in $\bigcup_{\alpha<\beta_p}\text{Sk}(\alpha)$, and in particular we can choose $x_p$ with $\text{ord}(x_p)<\beta_p$ such that $d(x_p,p)\leq 1/2$. If $\beta_p=\lambda+1$ for a countable ordinal $\lambda$, then by construction of $\text{Sk}(\omega_1)$ we have that $p$ belongs to the threading space $\text{Th}(x_p,y_p)$ for some $x_p,y_p\in \text{Sk}(\lambda)$. Since $p$ is in a thread of length $1$ with extremes $x_p,y_p$, the distance from $p$ to one of these extremes is less than or equal to $1/2$. Assume without loss of generality that $d(x_p,p)\leq 1/2$. We conclude that in any case there exists $x_p$ with $\text{ord}(x_p)<\text{ord}(p)$ such that $d(x_p,p)\leq 1/2$, and arguing in the same way there exists $x_q$ with $\text{ord}(x_q)<\text{ord}(q)$ such that $d(x_q,q)\leq 1/2$. 

The points $x_p,x_q$ verify that $\text{ord}(\{x_p,x_q\})<\beta$, so by inductive hypothesis there exists a sequence $\{x_k\}_{k=}^n\subset M$ with $x_0=x_p$ and $x_n=x_q$ such that $d(x_k,x_{k+1})\leq 1/2$ for all $0\leq k\leq n$. The result follows now adding the points $p$ and $q$ at the beginning and at the end of the sequence respectively.
\end{proof}
Let us mention that this previous lemma can be improved so that the distance between the points in the sequence is less than $1/4$, since this is the biggest possible gap in the threads we are considering. However, we do not consider this improvement to be relevant enough and prefer to prove it with a simpler and shorter argument, since we will only need to use the lemma as it is stated now. 

Now we can prove the first reduction result:
\begin{proposition}
\label{firstreduction}
Let $S$ be a closed subset of $ \text{Sk}(\omega_1)$ with at least two different points. If there exists $p\in S$ such that $p$ is isolated in $S$, then $S$ is not a Lipschitz retract of $ \text{Sk}(\omega_1)$.
\end{proposition}
\begin{proof}
Put $\varepsilon=d\big(p,S\setminus\{p\}\big)$, which is positive since $p$ is isolated in $S$. Suppose there exists a Lipschitz retraction $R\colon  \text{Sk}(\omega_1)\rightarrow S$, and put $K=\|R\|_\text{Lip}$.

Consider any point $q\in S$ different from $p$. By Lemma \ref{chaininmainspace} there exists a finite sequence $\{x_k\}_{k\in\mathbb{N}}\subset \text{Sk}(\omega_1)$ such that $x_0=p$ and $x_n=q$, and such that $d(x_k,x_{k+1})\leq 1/2$ for all $0\leq k\leq n-1$. By construction of the skein space $\text{Sk}(\omega_1)$, there exists an isometric copy of the threading space $\text{Th}(x_k,x_{k+1})$ in $\text{Sk}(\omega_1)$, so we may assume that these threading spaces are contained in $\text{Sk}(\omega_1)$. For every $0\leq k\leq n-1$, the threading space $\text{Th}(x_k,x_{k+1})$ itself contains the threads $T_{\gamma}(x_k,x_{k+1})$ for every $\gamma\in \Delta$. Choose $\gamma^*=(\gamma_i^*)_{i\in\mathbb{N}}\in \Delta$ such that $\gamma^*_i<\varepsilon/K$ for every $i\in\mathbb{N}$, and write $T^*_k$ to denote the thread $T_{\gamma^*}(x_k,x_{k+1})$ contained in the threading space $\text{Th}(x_k,x_{k+1})$ with extremes $x_k$ and $x_{k+1}$ for every $0\leq k\leq n-1$. 

Define
$$ k_0=\min\big\{k\in\{0,\dots,n-1\}\colon R\big(T^*_k\big)\nsubseteq \{p\}\big\},$$
which exists since $q\in T^*_{n-1}$ and $R(q)=q\in S\setminus\{p\}$. By definition of $k_0$, there exists a point $y_0\in T^*_{k_0}$ such that $R(y_0)\in S\setminus \{p\}$. If $k_0=0$, the point $y_0$ cannot be the lower extreme $x_0=p$ of the thread $T^*_{0}$, since $R(p)=p$. If $k_0\neq 0$, again we have that $y_0$ cannot be $x_{k_0}$ because $x_{k_0}$ is also in the previous thread $T^*_{k_0-1}$ as its higher extreme point, which would contradict the minimality of $k_0$. We conclude then that $R(x_{k_0})=p$. Since the gaps of $T^*_{k_0}$ are given by the sequence $\gamma^*$, they are all smaller than $\varepsilon/K$. We can then apply Remark \ref{threadandseparatedsets} to reach a contradiction with the existence of the retraction $R$.
\end{proof}

\subsection{Second reduction: points in limit ordinal generations}

In the construction of the skein space $\text{Sk}(\omega_1)$, we have a better understanding of the points belonging to successor ordinal generations than we do of points in limit ordinal generations. Indeed, for a point $p$ of order $\alpha+1$ we know that there exist two points $x$ and $y$, with at least one of them in generation $\alpha$ such that $p$ belongs to a thread $T_\gamma$ for a sequence $\gamma\in\Delta$ with extreme points $x$ and $y$. However, a point in a limit ordinal generation can initially only be described as a limit of a sequence of points in previous generations, and it does not belong to a thread or to any other defined structure. This subsection is dedicated to finding ways to describe these limit points in order to compensate for the comparatively low \emph{a priori} knowledge we have of them. 

For a closed subset $S$ of a metric space $M$ and $r>0$, define the \emph{open ball around $S$ of radius $r$}, denoted by $B(S,r)^\circ$ as the set
$$ B(S,r)^\circ=\{p\in M\colon d(p,S)< r\}.$$
The main result of this subsection is the following:

\begin{proposition}
\label{betaSkeinisRetractionofBall18}
In the skein space $Sk(\omega_1)$, for every ordinal number $\beta<\omega_1$, the $\beta$-skein $\text{Sk}(\beta)$ is a $1$-Lipschitz retraction of the ball $B(\text{Sk}(\beta),1/8)^\circ$.
\end{proposition}

In fact, as we are going to see, a stronger result is verified, which is helpful in the inductive argument we use to prove it.

Let us introduce some useful concepts: For an ordinal number $\beta<\omega_1$ and a point $p\in \text{Sk}(\omega_1)$, we may consider the set $P_\beta(p)=\big\{x\in \text{Sk}(\beta)\colon d(p,x)=d\big(p,\text{Sk}(\beta)\big)\big\}$. Since the $\beta$-skein $\text{Sk}(\beta)$ is not compact when $\beta>0$, we cannot easily ensure that $P_\beta(p)$ is nonempty in every case. In the case where $P_\beta(p)$ is nonempty for a point $p\in \text{Sk}(\omega_1)$ and an ordinal $\beta<\omega_1$, the members of $P_\beta(p)$ will be called the \emph{ancestors of $p$ of order $\beta$}.

If a point $p\in \text{Sk}(\omega_1)$ has order $\beta+1$ for some ordinal $\beta<\omega_1$, then it belongs to a threading space $\text{Th}(x,y)$ for a pair of points $(x,y)$ with $\text{ord}\{x,y\}=\beta$, and it is straightforward to see that the set of ancestors of $p$ of order $\beta$ is nonempty and is contained in $\{x,y\}$. Since every thread in $\text{Th}(x,y)$ has length $1$, if $d(p,\text{Sk}(\beta))<1/2$, then $P_\beta(p)$ is unique and is equal to either $x$ or $y$. The other point will be called the \emph{pseudo-ancestor of $p$ of order $\beta$}, and will be denoted by $Q_\beta(p)$. In this way, every point $p$ in a successor ordinal generation $G_{\beta+1}$ such that the distance $d(p,\text{Sk}(\beta))$ is smaller than $1/2$ will belong to the threading space $\text{Th}(P_\beta(p),Q_\beta(p))$. Notice that this concept is only defined for points in successor ordinal generations and with respect to the preceding ordinal.

For each ordinal number $\beta<\omega_1$, we say that a subset $S$ of $\text{Sk}(\omega_1)$ containing $\text{Sk}(\beta)$ is $\beta$-stable if for every point $p\in S$ there exists an ancestor $P_\beta(p)$ and it is unique, and moreover, the resulting well defined map $P_\beta\colon S\rightarrow M$ is a $1$-Lipschitz retraction. Hence, the main result of this subsection will be proven if we show that $B(\text{Sk}(\beta),1/8)^\circ$ is $\beta$-stable for all $\beta<\omega_1$. 

We prove the following even stronger result:

\begin{proposition}
\label{ultrastability}
Let $\beta<\omega_1$ be an ordinal number. If two points $p$ and $q$ belong to the ball $B(\text{Sk}(\beta),1/8)^\circ$, then the ancestors $P_\beta(p)$ and $P_\beta(q)$ exist and are unique. 
Moreover, if $P_\beta(p)\neq P_\beta(q)$, then $d(p,q)=d\big(p,P_\beta(p)\big)+d\big(P_\beta(p),P_\beta(q)\big)+d\big(P_\beta(q),q\big)$. 

In particular, the ball $B(\text{Sk}(\beta),1/8)^\circ$ is $\beta$-stable.
\end{proposition}
\begin{proof}
Put $\alpha=\text{ord}\{p,q\}$. We are going to prove the result by induction on $\alpha$. If $\alpha$ is smaller than $\beta$, then both $p$ and $q$ belong to the skein $\text{Sk}(\beta)$ and the result follows trivially. Hence, we will start the induction assuming $\alpha=\beta+1$. Let us divide the proof into three parts: the base case, the successor ordinal case, and the limit ordinal case. The base case is in fact the most technical part of the proof:

\vspace{2mm}
\textbf{1.- The base case: $\alpha=\beta+1$}
\vspace{2mm}

Suppose that $\alpha=\beta+1$. Since $\text{ord}\{p,q\}=\beta+1$, at least one of $p$ and $q$ is in generation $G_{\beta+1}$. Assume without loss of generality that $p$ belongs to generation $G_{\beta+1}$. As we discussed earlier, since the distance from $p$ to $\text{Sk}(\beta)$ is less than $1/8$ and in particular less than $1/2$, we have that the ancestor of order $\beta$ of $p$, $P_\beta(p)$, exists and is unique, and $p$ is in the threading space anchored at its ancestor and pseudo-ancestor of order $\beta$, denoted by $\text{Th}\big(P_\beta(p),Q_\beta(p)\big)$. Moreover, since $d(p,P_\beta(p))<1/4$, we have that $p$ is bound to $P_\beta(p)$ in $\text{Sk}(\omega_1)\setminus \text{Sk}(\beta)$ (we briefly discussed this when introducing the concept of boundness).
In other words, we have that the distance from $p$ to any point $x\in \text{Sk}(\beta)$ is computed by
\begin{equation}
\label{eq1basecase}
    d(p,x)=d\big(p,P_\beta(p)\big)+d\big(P_\beta(p),x\big)\qquad\text{for every }x\in \text{Sk}(\beta).
\end{equation}

Now, if the point $q$ is in the skein $\text{Sk}(\beta)$, then $q$ is its own ancestor of order $\beta$, and the result follows directly by the previous identity. Suppose then that $q$ is also in generation $G_{\beta+1}$. By the same discussion as above, $q$ belongs to the threading space $\text{Th}\big(P_\beta(q),Q_\beta(q)\big)$, and $q$ verifies the corresponding identity to \eqref{eq1basecase}. There are two possibilities: either both threading spaces $\text{Th}\big(P_\beta(p),Q_\beta(p)\big)$ and $\text{Th}\big(P_\beta(q),Q_\beta(q)\big)$ are the same, or $p$ and $q$ belong to different threading spaces. 

If $p$ and $q$ belong to two different threading spaces, then the result follows from equation \eqref{eq1basecase} (applied to both $p$ and $q$) and the construction of the skein $\text{Sk}(\beta+1)$. Otherwise, if both threading spaces $\text{Th}\big(P_\beta(p),Q_\beta(p)\big)$ and $\text{Th}\big(P_\beta(q),Q_\beta(q)\big)$ are the same, we may assume that the pseudo-ancestor of $p$, $Q_\beta(p)$, and the ancestor of $q$, $P_\beta(q)$, are the same point (otherwise $P_\beta(p)=P_\beta(q)$ and there is nothing left to prove). Now, on the one hand we have that $d\big(p,P_\beta(p)\big)<1/8$ and $d\big(q,P_\beta(q)\big)<1/8$ by hypothesis; and on the other hand the width of the threading spaces in the skein spaces we defined is less than $1/2$, so $d\big(P_\beta(p),P_\beta(q)\big)\leq 1/2$. Hence, we necessarily have that 
$$d\big(p,P_\beta(p)\big)+d\big(P_\beta(p),P_\beta(q)\big)+d\big(P_\beta(q),q\big)<3/4<|q-p|, $$
from which the result follows, whether $p$ and $q$ belong to the same thread in the threading space $\text{Th}\big(P_\beta(p),P_\beta(q)\big)$ or not.

\vspace{2mm}
\textbf{2.- The successor ordinal case: $\alpha=\eta+1$ for $\eta>\beta$}
\vspace{2mm}

Suppose now that $\alpha=\eta+1$ for some countable ordinal $\eta>\beta$, and that the result holds for every pair of points of order strictly less than $\eta+1$. Let us prove first that both ancestors $P_\beta(p)$ and $P_\beta(q)$ exist and are unique and that the ancestor operation commutes for $p$ and $q$ at order $\eta$, that is: $P_\beta\big(P_\eta(p)\big)=P_\beta(p)$ and $P_\beta\big(P_\eta(q)\big)=P_\beta(q)$. Since the argument is exactly the same for both points, we will only prove it for $p$, and again we may assume without loss of generality that $p$ is in the generation $G_{\eta+1}$. Since the distance from $p$ to $\text{Sk}(\beta)$ is less than $1/8$, we have as well that $d\big(p,\text{Sk}(\eta)\big)<1/8$ since $\text{Sk}(\beta)\subset \text{Sk}(\eta)$. Therefore, by the first step of the induction process we have that the ancestor of $p$ of order $\eta$ is unique and 
\begin{equation}
\label{eq1succord}
d(p,x)=d\big(p,P_\eta(p)\big)+d\big(P_\eta(p),x\big),\qquad\text{for all } x\in \text{Sk}(\eta).
\end{equation}
Moreover, now $P_\eta(p)\in \text{Sk}(\eta)$, and with the previous equation we can deduce that $P_\eta(p)$ belongs to the ball $B(\text{Sk}(\beta),1/8)^\circ$ as well, so by induction again we have that $P_\beta\big(P_\eta(p)\big)$ is unique, and $d\big(P_\eta(p),x\big)=d\big(P_\eta(p),P_\beta\big(P_\eta(p)\big)\big)+d\big(P_\beta\big(P_\eta(p)\big),x\big)$. These two identities result in the following equation:
\begin{equation*}
    d(p,x)=d\big(p,P_\eta(p)\big)+d\big(P_\eta(p),P_\beta\big(P_\eta(p)\big)\big)+d\big(P_\beta\big(P_\eta(p)\big),x\big).
\end{equation*}

Applying equation \eqref{eq1succord} for $P_\beta\big(P_\eta(p)\big)\in \text{Sk}(\beta)$ we can put the first two terms of the right hand side in the previous equation as $d\big(p,P_\eta(p)\big)+d\big(P_\eta(p),P_\beta\big(P_\eta(p)\big)\big)=d\big(p,P_\beta\big(P_\eta(p)\big)\big)$, and finally obtain:
\begin{equation}
    \label{eq2succord}
    d(p,x)=d\big(p,P_\beta\big(P_\eta(p)\big)\big)+d\big(P_\beta\big(P_\eta(p)\big),x\big),\qquad\text{for all } x\in \text{Sk}(\beta).
\end{equation}
Now, from equation \eqref{eq2succord} it is easy to prove that $P_\beta(p)$ is unique and $P_\beta\big(P_\eta(p)\big)=P_\beta(p)$.

To finish with this case, suppose that $P_\beta(p)\neq P_\beta(q)$. Then $P_\eta(p)\neq P_\eta(q)$ by what we just proven. We can now apply the inductive hypothesis several times and deduce that:
\begin{align*}
    d(p,q)&=d\big(p,P_\eta(p)\big)+d\big(P_\eta(p),P_\eta(q)\big)+d\big(P_\eta(q),q\big)\\
    &= d\big(p,P_\eta(p)\big)+d\big(P_\eta(p),P_\beta(p)\big)+\big(P_\beta(p),P_\beta(q)\big)+\big(P_\beta(q),P_\eta(q)\big)+\big(P_\eta(q),q\big)\\
    &=d\big(p,P_\beta(p)\big)+\big(P_\beta(p),P_\beta(q)\big)+\big(P_\beta(q),q\big).
\end{align*}
This finishes the successor ordinal case.

\vspace{2mm}
\textbf{3.- The limit ordinal case}
\vspace{2mm}

Suppose finally that $\alpha$ is a limit ordinal. As in the previous case, we start by proving that $P_\beta(p)$ and $P_\beta(q)$ exist and are unique. Similarly, we assume that $\text{ord}(p)=\alpha$, and we only prove it for $p$. Consider a sequence $\{p_n\}_{n\in\mathbb{N}}$ of points in $\text{Sk}(\alpha)$ convergent to $p$ and such that $\text{ord}(p_n)<\alpha$ for all $n\in\mathbb{N}$. Moreover, since the ball $B(\text{Sk}(\beta),1/8)^\circ$ is an open set of $\text{Sk}(\omega_1)$ which contains $p$, we may suppose that $d\big(p_n,\text{Sk}(\beta)\big)<1/8$ as well for all $n\in\mathbb{N}$. Therefore, by inductive hypothesis, $P_\beta(p_n)$ is unique for all $n\in\mathbb{N}$, and 
\begin{equation}
\label{eq1limitord}
    d(p_n,x)=d\big(p_n,P_\beta(p_n)\big)+d\big(P_\beta(p_n),x\big),\quad\text{for all }x\in \text{Sk}(\beta)\text{ and all }n\in\mathbb{N}.
\end{equation}
We are going to prove first that the sequence $\{P_\beta(p_n)\}_{n\in\mathbb{N}}$ is convergent. Indeed, since $\text{ord}(p_n)<\alpha$ for all $x\in\mathbb{N}$, for all $n,m\in \mathbb{N}$ such that $P_\beta(p_n)\neq P_\beta(x_m)$, we have that $d(p_n,p_m)=d\big(p_n,P_\beta(p_n)\big)+d\big(P_\beta(p_n),P_\beta(p_m)\big)+d\big(P_\beta(p_m),p_m\big)$. In particular, $d\big(P_\beta(p_n),P_\beta(p_m)\big)\leq d(p_n,p_m)$ for all $n,m\in\mathbb{N}$, whether $P_\beta(p_n)\neq P_\beta(x_m)$ or not. Since the sequence $\{p_n\}_{n\in\mathbb{N}}$ converges, it is a Cauchy sequence, which implies that the sequence $\{P_\beta(p_n)\}_{n\in\mathbb{N}}$ is a Cauchy sequence as well, and thus convergent in the complete metric space $\text{Sk}(\beta)$. Denote the limit of this sequence by $P^*$, which belongs to the set $\text{Sk}(\beta)$. 

Taking the limit when $n$ tends to infinity in equation \eqref{eq1limitord}, we obtain that 
$$ d(p,x)=d\big(p,P^*\big)+d\big(P^*,x\big),\qquad\text{for all } x\in \text{Sk}(\beta).$$
Similarly to the successor ordinal case, from this equation it follows that $P_\beta(p)=P^*$ and it is unique. 

Now, suppose that $P_\beta(p)\neq P_\beta(q)$, and consider two sequences $\{p_n\}_{n\in\mathbb{N}}$ and $\{q_n\}_{n\in\mathbb{N}}$ in $B(\text{Sk}(\beta),1/8)^\circ$ converging to $p$ and $q$ respectively, and such that $\text{ord}\{p_n,q_n\}<\alpha$. By the previous argument, we have that:
\begin{align*}
    d(p,q)&=\lim_{n}d(p_n,q_n)=\lim_{n}\big(d\big(p_n,P_\beta(p_n)\big)+d\big(P_\beta(p_n),P_\beta(q_n)\big)+d\big(P_\beta(q_n),q_n\big)\\
    &=d\big(p,P_\beta(p)\big)+d\big(P_\beta(p),P_\beta(q)\big)+d\big(P_\beta(q),q\big),
\end{align*}
which concludes the proof.
\end{proof}

\begin{figure}
\includegraphics[width=\textwidth]{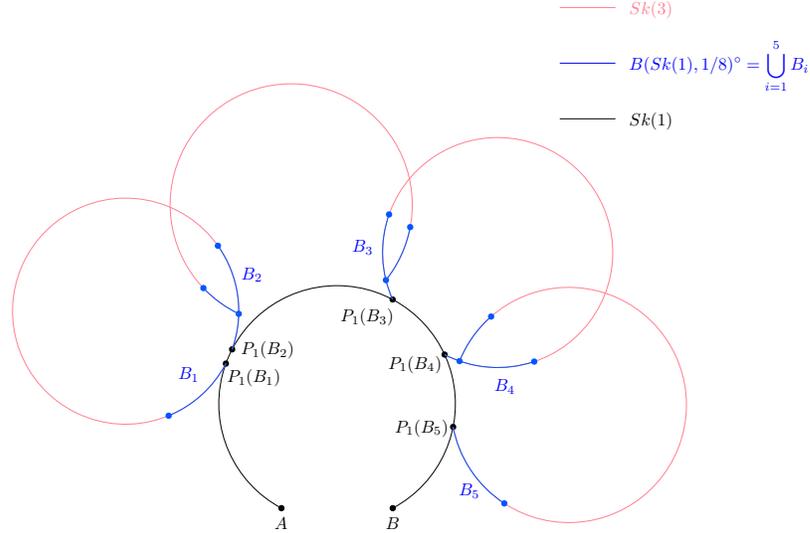}
\caption{The skein $\text{Sk}(1)$ is a $1$-Lipschitz retract of the ball $B(\text{Sk}(1),1/8)$ with the ancestor map.}
\label{UltrastabilityFigure}
\end{figure}

In Figure \ref{UltrastabilityFigure} we observe conceptually Proposition \ref{ultrastability}. In this diagram we portray again a subset of the skein $\text{Sk}(3)$, and the ball $B(\text{Sk}(1),1/8)^\circ$ (colored in blue) is partitioned into 5 subsets $\{B_i\}_{i=1}^5$ such that every point in the same $B_i$ has the same ancestor of order $1$. The ancestor map clearly defines in this case a $1$-Lipschitz retraction onto $\text{Sk}(1)$.

Finally, with this proposition, the second reduction result follows directly:
\begin{proof}[Proof of \ref{betaSkeinisRetractionofBall18}]
It follows directly from Proposition \ref{ultrastability}.
\end{proof}
In the proof of the main theorem, given a separable subset, we will consider a bigger separable subset that is in some sense closed for the operation of taking ancestors closer than $1/8$. Specifically, we have the following Lemma:

\begin{lemma}
\label{EngordarS}
Given a separable subset $S$ of the metric space $\text{Sk}(\omega_1)$, there exists a separable set $\widehat{S}\subset \text{Sk}(\omega_1)$ containing $S$ such that: for every point $x\in \widehat{S}$ and every ordinal $\beta<\omega_1$ such that $d(x,\text{Sk}(\beta))<1/8$, the unique ancestor of order $\beta$ of $x$ belongs to $\widehat{S}$.
\end{lemma}
\begin{proof}
For any point $x\in \text{Sk}(\omega_1)$, put $\beta_0(x)=\text{ord}(x)$, which is a countable ordinal. Trivially we have that the ancestor $P_{\beta_0}(x)=x$ is unique. We might define then $\beta_1(x)=\min\{\beta<\omega_1\colon d\big(P_{\beta_0}(x),\text{Sk}(\beta)\big)<1/8\}$, which verifies that $P_{\beta_1}\big(P_{\beta_0}(x)\big)$ is unique as well.

Then, we can inductively define a decreasing sequence of ordinal numbers 
$$\beta_n(x)=\min\{\beta<\omega_1\colon d\big(\big(P_{\beta_{n-1}(x)}\circ\dots\circ P_{\beta_0(x)}\big)(x),\text{Sk}(\beta)\big)<1/8\}$$
for each $n\in\mathbb{N}$. Since $\beta_{n+1}(x)\leq \beta_n(x)$ for every $n\in\mathbb{N}$ and the ordinal numbers are well ordered, there must exist $n_0(x)\in\mathbb{N}$ such that $\beta_n(x)=\beta_{n_0}(x)$ for all $n\geq n_0$. 

Now, given a separable subset $S$ of the metric space $M$, take $D$ a countable and dense subset of $S$. Consider the set $\widehat{D}$ defined by:
$$ \widehat{D}=\bigcup_{x\in D}\bigcup_{n\in\mathbb{N}}\bigcup_{\beta=\beta_{n}(x)}^{\beta_{n-1}(x)}P_\beta\big(\big(P_{\beta_{n-1}(x)}\circ\dots\circ P_{\beta_0(x)}\big)(x)\big),$$
which is countable, contains $D$, and verifies that for any point $x\in \widehat{D}$ and any ordinal $\beta$ with $d(x,\text{Sk}(\beta))<1/8$, the ancestor $P_\beta(x)$ belongs to $\widehat{D}$ as well. 

Finally, put $\widehat{S}=\overline{\widehat{D}}$. The set $\widehat{S}$ is separable and it contains $S$. For any point $x\in \widehat{S}$ and any ordinal $\beta<\omega_1$ such that $x$ belongs to the ball $B(\text{Sk}(\beta),1/8)^\circ$, we have that $x$ is the limit of a sequence $\{x_n\}_{n\in\mathbb{N}}$ of points in $\widehat{D}$ which are also in $B(\text{Sk}(\beta),1/8)^\circ$. As we argued in the proof of Proposition \ref{ultrastability}, we have that the sequence $\{P_\beta(x_n)\}_{n\in\mathbb{N}}$, which is contained in $\widehat{D}$, converges to $P_\beta(x)$. The statement of the remark now follows directly.
\end{proof}

We will use as well the fact that every separable subset of the skein $\text{Sk}(\omega_1)$ is contained in the closure of the union of countably many threads:

\begin{lemma}
\label{Containedincountable}
Let $S$ be a separable subset of the skein $\text{Sk}(\omega_1)$. Then there exists a countable family of pairs $\{(x_n,y_n)\}_{n\in\mathbb{N}}$ in $\text{Sk}(\omega_1)\times\text{Sk}(\omega_1)$ and a countable family of sequences $\{\gamma^n\}_{\in\mathbb{N}}$ in $\Delta$ such that the following property is verified:

For every point $x\in S$ in a successor ordinal generation there exists a natural number $n_x$ such that $x$ belongs to the interior of the thread $T_{\gamma^{n_x}}(x_{n_x},y_{n_x})\subset \text{Sk}(\omega_1)$.
\end{lemma}
\begin{proof}
Suppose by contradiction that the result fails. Since $S$ is separable, there are only countably many ordinals $\alpha<\omega_1$ such that the intersection of $S$ with generation $G_{\alpha}$ is nonempty. Hence, there must exist one successor ordinal $\alpha_0+1$ such that for every countable collection of pairs $\big\{\{x_n,y_n\}\big\}_{n\in\mathbb{N}}$ in $\Gamma_{\alpha_0}$ and any countable family of sequences $\{\gamma^n\}_{\in\mathbb{N}}$ in $\Delta$, there exists a point in $S\cap G_{\alpha_0+1}$ that lies outside of the interior of the thread $T_{\gamma^n}(x_n,y_n)\subset \text{Sk}(\alpha_0+1)$ for all $n\in\mathbb{N}$. 

Since every point in $S\cap G_{\alpha_0+1}$ belongs to the interior of a thread anchored at a pair of points in $\Gamma_{\alpha_0}$, by a standard transfinite induction argument we may find an uncountable set of different points $\{p_i\}_{i\in I}$ in $S\cap G_{\alpha_0+1}$ and an uncontable family of different threads $\{T_{\gamma^i}(x_i,y_i)\}_{i\in I}$ with $\{x_i,y_i\}\in \Gamma_{\alpha_0}$ and $\gamma^i\in \Delta$ for all $i\in I$, such that $p_i$ belongs to the interior of the thread $T_{\gamma^i}(x_i,y_i)$ for all $i\in I$. However, this implies that the family $\{T_{\gamma^i}(x_i,y_i)^\circ \cap S\}_{i\in I}$ is an uncountable collection of nonempty and open subsets of $S$ which are pairwise disjoint, which contradicts the separability of $S$.
\end{proof}

\subsection{Proving the general case}

We proceed now to prove the main result of this article:
\begin{proof}[Proof of Theorem \ref{maintheorem}]
Consider the complete skein $\text{Sk}(\omega_1)$. We will prove that it does not contain any non-singleton separable Lipschitz retracts. 

We will proceed by contradiction. Let $S$ be a separable subset of $\text{Sk}(\omega_1)$ containing at least two points. We may assume that $S$ has no isolated points by Proposition \ref{firstreduction}. Suppose there exists a Lipschitz retraction $R\colon \text{Sk}(\omega_1)\rightarrow S$ onto $S$, and put $K=\|R\|_\text{Lip}$. We are going to find a specific thread $T^*$ in $\text{Sk}(\omega_1)$ such that when restricting the map $R$ to $T^*$, the resulting $K$-Lipschitz function can be transformed to yield a contradiction. Because of the length of the proof, we divide it in two parts: The first part describes the process to define the problematic thread $T^*$, while the second part deals with the map $R_{|T^*}$, and how to transform it to arrive at a contradiction. We will also highlight important facts throughout the proof to help in its readability.

\vspace{2mm}
\textbf{1.- Defining the conflicting thread $T^*$}
\vspace{2mm}

We start by finding two points to anchor th thread $T^*$:

\begin{fact}
    There exist two points $p$ and $q$ in $S$ closer than $1/2$, and such that there exists a successor ordinal $\beta_0$ with $p,q\in B(\text{Sk}(\beta_0),1/8)^\circ$ which verifies that $P_{\beta_0}(p)\neq P_{\beta_0}(q)$. 
    
    Moreover, the ancestor $P_{\beta_0}(p)$ is in generation $\beta_0$, and is contained in a thread $T_{\gamma^0}\big(A_0,B_0\big)$ with $\big\{A_0,B_0\big\}\in \Gamma_{\beta_0-1}$.
\end{fact}

\begin{proof}[Proof of Fact 1]

Define $\alpha_0$ as the least ordinal such that $S\cap \text{Sk}(\alpha_0)$ is nonempty,

We divide the proof in two cases:

\textbf{Case 1}: There exist two different points $p$ and $q$ in $B(\text{Sk}(\alpha_0),1/8)^\circ\cap S$ closer than $1/2$ such that $P_{\alpha_0}(p)\neq P_{\alpha_0}(q)$.

In this case, if the ordinal $\alpha_0$ is a successor ordinal, putting $\beta_0=\alpha_0$ we are done, since by definition of $\alpha_0$ both $P_{\alpha_0}(p)$ and $P_{\alpha_0}(q)$ belong to generation $G_{\alpha_0}$. 

Suppose then that $\alpha_0$ is a limit ordinal. Since $p$ and $q$ belong to the ball $B(\text{Sk}(\alpha_0),1/8)^\circ$, the ordinal 
$$\alpha_1 = \min\{\beta<\omega_1\colon \{p,q\}\subset B(\text{Sk}(\beta),1/8)^\circ\} $$
is less than $\alpha_0$. Both $P_{\alpha_1}(p)$ or $P_{\alpha_1}(q)$ are well defined and unique. Moreover, by minimality, $\alpha_0$ must be a successor ordinal, and at least one of $P_{\alpha_1}(p)$ or $P_{\alpha_1}(q)$ must belong to generation $G_{\alpha_1}$. Therefore we can put $\beta_0=\alpha_1$ and Fact 1 is proven for Case 1.

\textbf{Case 2}: There exists a point $A\in S\cap \text{Sk}(\alpha_0)$ such that for all $x\in B(A,1/8)^\circ\cap S$ we have that $P_{\alpha_0}(x)=A$. 

Notice that the above statement follows from negating the assumption of Case 1. In this case define the ordinal
$$ \eta_0 =\min\{\eta<\omega_1\colon P_{\eta}(x)\neq A\text{ for some }x\in B(A,1/8)^\circ\cap S\},$$

Such an ordinal number must exist since $A$ is not isolated in $S$ by assumption. Moreover, $\eta_0$ must be a successor ordinal since every point in a limit ordinal generation is the limit of the succession given by its previous (existing) ancestors. Take now any point $p\in B(A,1/8)^\circ\cap S$ such that $P_{\eta_0}(p)\neq A$, and set $q =A$. Putting $\beta_0=\eta_0$, we have that both $p$ and $q$ belong to $B(\text{Sk}(\beta_0),1/8)^\circ$ and $P_{\beta_0}(p)\neq P_{\beta_0}(q)=q$. Moreover, the ancestor $P_{\beta_0}(p)$ belongs to generation $G_{\beta_0}$ by minimality.
\end{proof}

With this in mind, we can apply Proposition \ref{totallyseparated} to the thread $T_{\gamma^0}(A_0,B_0)$ and the point $P_{\beta_0}(p)$, to find a compact subset $C_0\subset T_{\gamma^0}(A_0,B_0)$ with diameter less than $d\big(P_{\beta_0}(p),P_{\beta_0}(q)\big)$ such that $P_{\beta_0}(p)\in C_0$ and $C_0$ is open and closed in $\text{Sk}(\beta_0)$. Put $C_1=\text{Sk}(\beta_0)\setminus C_0$. Then the point $P_{\beta_0}(q)$ belongs to $C_1$, and since $C_0$ is compact and disjoint from the closed set $C_1$, the distance $d\big(C_0,C_1\big)$ is strictly positive. Put $d_0=d\big(C_0,C_1\big)>0$.

Figure \ref{SeparatingPbetas} summarizes one possible layout of the elements we have defined so far in the skein $\text{Sk}(\beta_0)$.

\begin{figure}
    \includegraphics[width=\textwidth]{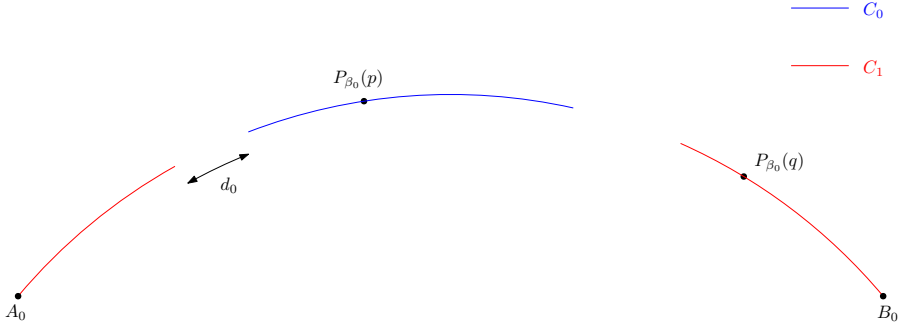}
    \caption{Thread of the skein $\text{Sk}(\beta_0)$ showing one possible arrangement of the ancestors of $p$ and $q$ and choice of $C_0$ and $C_1$.}
    \label{SeparatingPbetas}
\end{figure}

Now, the separation between the sets $C_0$ and $C_1$ allows us to use Remark \ref{threadandseparatedsets} to obtain the following fact:

\begin{fact}
\label{Finalthread}
If $T=[0,l]$ is a thread whose gaps are all smaller than $d_0/2K$, there cannot be any $2K$-Lipschitz map $F\colon T\rightarrow \text{Sk}(\beta_0)$ such that $F(0)=P_{\beta_0}(p)\in C_0$ and $F(l)=P_{\beta_0}(q)\in C_1$. 
\end{fact}

Next, we define the subset $\widehat{S}\subset \text{Sk}(\omega_1)$ using Lemma \ref{EngordarS} such that the following fact is verified:
\begin{fact}
\label{widehatScontainsancestry}
The set $\widehat{S}\subset \text{Sk}(\omega_1)$ is separable, it contains the set $S$, and for any point $x\in \widehat{S}$ and any ordinal number $\beta$ such that $x\in B(\text{Sk}(\beta),1/8)$, the unique ancestor $P_\beta(x)$ belongs to $\widehat{S}$.
\end{fact}

To continue with the proof, since $\widehat{S}$ is separable, by Lemma \ref{Containedincountable} we can find a countable family of sequences $\{\gamma^n\}_{n\in\mathbb{N}}$ in $\Delta$, and a countable set of pairs $\{(x_n,y_n)\}_{n\in\mathbb{N}}$ in $\text{Sk}(\omega_1)\times \text{Sk}(\omega_1)$ such that, denoting by $T^n$ the thread $T_{\gamma^n}(x_n,y_n)\subset \text{Sk}(\omega_1)$ for each $n\in\mathbb{N}$, the countable family of threads $\mathcal{T}_0=\{T^n\}_{n\in\mathbb{N}}$ verifies that any point $x\in \widehat{S}$ belonging to a successor ordinal generation is contained in the interior of at least one thread $T^{n_x}$ for some $n_x\in\mathbb{N}$. For every $n\in\mathbb{N}$, the thread $T^n=T_{\gamma^n}(x_n,y_n)\in \mathcal{T}_0$ has length $1$, and so the open subsets given by $[x_n,x_n+1/8)_{T^n}$ and $(y_n-1/8,y_n]_{T^n}$ are separable subsets that do not intersect. Define for every $n\in\mathbb{N}$ two countable sets $D^n_1$ and $D^n_2$ such that $D^n_1$ is dense in $[x_n,x_n+1/8)_{T^n}$ and $D^n_2$ is dense in $(y_n-1/8,y_n]_{T^n}$. 

Finally, we can define the countable family of threads given by 
$$ \mathcal{T}=\bigcup_{n\in\mathbb{N}}\Bigg(\bigcup_{(x,y)\in D^n_1\times D^n_2}\big\{[x,y]_{T^n}\big\}\Bigg).$$
Notice that each thread in $\mathcal{T}_0$ has Lebesgue measure of at least $1/2$. Therefore, the measure of the threads in $\mathcal{T}$ is bounded below by $1/4$. We can apply now Theorem $\ref{nolipschitzmap}$ with $\varepsilon=1/4$ and $2K\geq 1$ to find a sequence $\gamma^* \in \Delta$ with the following property:

\begin{fact}
\label{choicesequence}

There exists a sequence $\gamma^* = \{\gamma^*_k\}_{k\in\mathbb{N}}\in \Delta$ such that:

For any thread $S$ of length $l_S$ whose sequence of gaps $\{C^S_k\}_{k\in\mathbb{N}}$ in decreasing length order verifies $\text{length}(C^S_k)<\gamma^*_k$ for all $k\in\mathbb{N}$, it holds that: For every $K'\leq 2K$, if there exists a $K'$-Lipschitz function $F\colon S\rightarrow [x,y]_{T^n}$ such that $F(0)=x$ and $F(l_s)=y$, where $n\in\mathbb{N}$ and $(x,y)\in D^n_1\times D^n_2$; then the thread $S$ has a gap of length greater than or equal to $d(x,y)/K'$.

Moreover, without loss of generality we may choose $\gamma^*$ such that  
$$\gamma^*_k<\min\bigg\{\frac{1}{16K},\frac{d_0}{2K}\bigg\}$$
for all $k\in\mathbb{N}$.
\end{fact}

Since the sequence $\gamma^*$ belongs to $\Delta$, the associated thread $T_{\gamma^*}(p,q)$ belongs to the threading space $\text{Th}(p,q)$, and is therefore a subset of $\text{Sk}(\omega_1)$. Put $T^*=T_{\gamma^*}(p,q)$. This is the problematic thread we will use to reach a contradiction. Recall that the length of the gaps of the thread $T^*$ is given by the sequence $\gamma^*\in\Delta$. Hence, we have the following result by the choice of $\gamma^*$ and Fact \ref{Finalthread}:
\begin{fact}
\label{choicethread}
There does not exist any $2K$-Lipschitz map $F\colon T^*\rightarrow \text{Sk}(\beta_0)$ such that $F(p)=P_{\beta_0}(p)$ and $F(q)=P_{\beta_0}(q)$.
\end{fact}

In the next section we will find a function in contradiction with this last fact. The retraction $R$ from $\text{Sk}(\omega_1)$ onto $S$ can be restricted to $T^*$ to obtain a $K$-Lipschitz map $R_{|T^*}\colon T^*\rightarrow \widehat{S}$ such that $R(p)=p$ and $R(q)=q$. This restriction will be the starting point in the process to define the contradicting function. 

\vspace{2mm}
\textbf{2.- Transforming the map $R_{|T^*}$}
\vspace{2mm}

We can only ensure that the image of the map $R_{|T^*}$ is contained in $S$, and so the order of $R_{|T^*}(T^*)$ is less than the order of $S$, but it can still be higher than $\beta_0$. We are going to transform inductively the map $R_{|T^*}$ to reduce the order of its image until we arrive at $\beta_0$, where we will reach a contradiction. 

In order to do this, we need the following result:

\begin{claim}
\label{reduceorderlemma}
Let $T^*,~ \widehat{S},~ K$ and $\beta_0$ be defined as above. Let $F\colon T^*\rightarrow \widehat{S}$ be a Lipschitz map such that $\|F\|_\text{Lip}<2K$, and $F(p)=P_\beta(p)$ and $F(q)=P_\beta(q)$ for some ordinal $\beta\geq \beta_0$. Then we have the three following results:
\begin{enumerate}[label=(\Alph*),ref=(\Alph*)]
    \item If $\text{ord}(F(T^*))$ is a limit ordinal then there exists an ordinal $\widehat{\beta}\geq \beta_0$, and a $K$-Lipschitz function $\widehat{F}\colon T^*\rightarrow \widehat{S}$ such that $\widehat{F}(p)=P_{\widehat{\beta}}(p)$, $\widehat{F}(q)=P_{\widehat{\beta}}(q)$, and $\text{ord}(\widehat{F}(T^*))<\text{ord}(F(T^*))$.
    \item If $\text{ord}(F(T^*))$ is a successor ordinal $\alpha+1$ such that $\beta<\alpha+1$ then for every $\varepsilon>0$ such that $\|F\|_\text{Lip}+\varepsilon< 2K$, there exists a Lipschitz function $\widehat{F}\colon T^*\rightarrow \widehat{S}$ such that $\|\widehat{F}\|_\text{Lip}<\|F\|_\text{Lip}+\varepsilon$, $\widehat{F}(p)=P_{\beta}(p)$, $\widehat{F}(q)=P_\beta(q)$, and $\text{ord}(\widehat{F}(T^*))\leq\alpha<\text{ord}(F(T^*))$.
    \item If $\text{ord}(F(T^*))$ is a successor ordinal $\alpha+1$ such that $\beta=\alpha+1$, and $\beta>\beta_0$, then for every $\varepsilon>0$ such that $\|F\|_\text{Lip}+\varepsilon< 2K$, there exists a Lipschitz function $\widehat{F}\colon T^*\rightarrow\widehat{S}$ such that $\|\widehat{F}\|_\text{Lip}<\|F\|_\text{Lip}+\varepsilon$, $\widehat{F}(p)=P_\alpha(p)$, $\widehat{F}(q)=P_\alpha(q)$ and $\text{ord}(\widehat{F}(T^*))\leq\alpha<\text{ord}(F(T^*))$.
\end{enumerate}
\end{claim}
Before proving this claim let us discuss its implications: The map $R_{|T^*}$ verifies the general hypothesis of the claim with $\beta=\text{ord}\{p,q\}$. Notice as well that if a function $F$ verifies 
either of the conditions $(A)$, $(B)$ or $(C)$ then the resulting map $\widehat{F}$ for any valid $\varepsilon>0$ verifies again the general conditions of the claim. In all three cases, the order of the image of the map $\widehat{F}$ produced is an ordinal strictly lower than the order of the image of $F$. 

This means that putting $F_0=R_{|T^*}$, we can define inductively $K$-Lipschitz maps $\{F_n\}_{n\in\mathbb{N}}$ such that 
$$ 
F_{n+1}=
\begin{cases}
\widehat{F_n}&\text{ if }F_n \text{ verifies (A), (B) or (C)},\\
F_n&\text{ otherwise }.
\end{cases}
$$
We may choose any valid $\varepsilon>0$ at the steps that require it.
There must exist $n_0\in\mathbb{N}$ such that $F_n=F_{n_0}$ for all $n\geq n_0$. Indeed, otherwise the sequence $\{\text{ord}(F_n)(T^*)\}_{n\in\mathbb{N}}$ is an infinite strictly decreasing sequence of ordinal numbers, which are well ordered, resulting in a contradiction. 

Therefore, the map $F_{n_0}\colon T^*\rightarrow \widehat{S}$ is a Lipschitz map with $\|F_{n_0}\|_\text{Lip}<2K$ such that $F(p)=P_\beta(p)$ and $F(q)=P_\beta(q)$ for some $\beta\geq \beta_0$ that does not verify neither $(A),~(B)$ nor $(C)$. Since it does not verify $(A)$, we have that $\text{ord}(F_{n_0})(T^*)$ is successor ordinal $\alpha+1$. This means that in order to fail $(B)$, the ordinal $\alpha+1$ must equal $\beta$. In turn, since $F_{n_0}$ does not meet the requirements of $(C)$ either, we conclude that $\beta$ (that is, the ordinal such that $F_{n_0}(p)=P_{\beta}(p)$, $F_{n_0}(q)=P_{\beta}(q)$, and the order of $F_{n_0}(T^*)$) equals $\beta_0$. 

In conclusion, $F_{n_0}\colon T^*\rightarrow \widehat{S}\cap \text{Sk}(\beta_0)$ is a Lipschitz map with $\|F_{n_0}\|_\text{Lip}<2K$ from the thread $T^*=T_{\gamma^*}(p,q)$ into $\text{Sk}(\beta_0)$ such that $F_{n_0}(p)=P_{\beta_0}(p)$ and $F_{n_0}(q)=P_{\beta_0}(q)$. This contradicts Fact \ref{choicethread}, which leads to the desired contradiction. It only remains to prove Claim \ref{reduceorderlemma}.

\begin{proof}[Proof of Claim \ref{reduceorderlemma}]
We prove each statement separately:

\begin{proof}[Proof of (A)]
Put $\alpha=\text{ord}(F(T^*))$. Then $\alpha\geq \beta$. Since $T^*$ is compact, the image $F(T^*)$ is also compact in $\text{Sk}(\omega_1)$. Therefore, there exists a point $x_0\in T^*$ such that $F(x_0)$ belongs to the limit generation $G_\alpha$ in $\text{Sk}(\omega_1)$. Put $r_0=\min\{d(F(x_0),\text{Sk}(\beta_0)),1/8\}$. Since $\beta_0$ is a successor ordinal, the number $r_0$ is strictly positive. 

Now, choose a finite set of points $\{x_i\}_{i=1}^n\subset T^*$ such that $F(T^*)\subset \bigcup_{i=1}^n B(F(x_i),r_0/2)$. For each $i=1,\dots,n$, the ordinal number
$$ \alpha_i=\min\{\alpha<\omega_1\colon d\big(F(x_i),\text{Sk}(\alpha)\big)<r_0/2\} $$
is a successor ordinal strictly smaller than the order of $F(T^*)$. Hence, if $\widehat{\beta}=\max\{\alpha_i\colon i=1,\dots,n\}$, we have that $\widehat{\beta}<\text{ord}(F(T^*))$ and $d\big(F(x),\text{Sk}(\widehat{\beta})\big)<r_0$ for all $x\in T^*$. This implies that $\widehat{\beta}$ is greater than $\beta_0$. Applying Corollary \ref{betaSkeinisRetractionofBall18}, since $r_0<1/8$, we have that the map 
\begin{align*}
    \widehat{F}\colon T^*&\longrightarrow \widehat{S}\\
    x&\longmapsto P_{\widehat{\beta}}\big(F(x)\big)
\end{align*}
is a $K$-Lipschitz map. It is well defined since the ancestor of order $\widehat{\beta}$ of each point in $F(T^*)$ is unique, and the image of any point $x\in T^*$ belongs to the set $\widehat{S}$ again since $d\big(F(x),\text{Sk}(\widehat{\beta})\big)<r_0$ for all $x\in T^*$ (see Fact \ref{widehatScontainsancestry}). We have then that $\widehat{F}(p)=P_{\widehat{\beta}}(P_\beta(p))=P_{\widehat{\beta}}(p)$, and similarly $\widehat{F}(q)=P_{\widehat{\beta}}(q)$. Finally, the order of $\widehat{F}(T^*)$ is at most $\widehat{\beta}$ as well, so it is verified that $\text{ord}(\widehat{F}(T^*))<\text{ord}(F(T^*))$.
\end{proof}
\begin{proof}[Proof of (B)]
Put $\overline{K}=\|F\|_\text{Lip}$. Suppose that the order of $F(T^*)$ is $\alpha+1$, and that $\beta<\alpha+1$. Since the image of $F$ is in $\widehat{S}$ and $\alpha+1$ is a successor ordinal, there exists a subsequence $\{n_k\}_{k\in\mathbb{N}}$ such that $F(T^*)\cap G_{\alpha+1}=F(T^*)\setminus \text{Sk}(\alpha)$ is contained in the union of threads $\bigcup_{k\in\mathbb{N}} T^{n_k}$. Informally, the ``problematic" part of $F(T^*)$ is contained in this countable set of threads (without the extreme points, since these always belong to a lower generation), which is a subfamily of the set $\mathcal{T}_0$ we have considered in the definition of $T^*$.

Hence, for any $t\in T^*$ such that $F(t)\in T^{n_{k(t)}}$ for some $n_{k(t)}\in\mathbb{N}$, we may find an extended interval $I_t$ in  $T^*$ containing $t$ such that $I_t$ is maximal for $F$ and $T^{n_{k(t)}}$. The extended interval $I_t$ is actually of the form $[a_t,b_t]_{T^*}$ with $p\leq a_t<b_t\leq q$, since if $I_t$ contains both extremes $p$ and $q$ of $T^*$, then necessarily $\{F(p),F(q)\}=\{P_\beta(p),P_\beta(q)\}\in \Gamma_{\alpha}$, so $\alpha=\beta$, and thus $I_t=[p,q]_{T^*}=T^*$. 

With this idea, since $F(T^*)$ is separable, we can define a countable family of maximal intervals $\{[a_i,b_i]_{T^*}\}_{i\in\mathbb{N}}$ in $T^*$ such that $F\big([a_i,b_i]_{T^*}\big)$ is contained in $T^{n_{k(i)}}$ for all $i\in\mathbb{N}$, and every point $t\in T^*$ such that its image $F(t)$ is in generation $G_{\alpha+1}$ is contained in $[a_i,b_i]_{T^*}$ for some $i\in\mathbb{N}$. To simplify the notation, we abuse it and write $n_{k(i)}=i$. Therefore, we will write that $F\big([a_i,b_i]_{T^*}\big)$ is contained in the thread $T^i$. Recall that the thread $T^i\in \mathcal{T}_0$ belongs to the threading space $\text{Th}(x_i,y_i)$ for every $i\in\mathbb{N}$. Again informally, we have identified a countable family of maximal intervals in $T^*$ that contain all the points whose image we need to change to prove (B). 

In the following Fact, we ``correct" the image of this countable family of maximal intervals.

\begin{fact}
    For every $i\in\mathbb{N}$, there exists a Lipschitz function $F_i\colon T^*\rightarrow \widehat{S}$ with $\|F_i\|_\text{Lip}\leq \|F\|+\varepsilon$ such that $F_i(t)=F(t)$ for all $t\in T^*\setminus [a_i,b_i]_{T^*}$ and $F_i(t)\in \{x_{i},y_{i}\}$ for all $t\in [a_i,b_i]_{T^*}$.
\end{fact}
\begin{proof}

Fix $i\in\mathbb{N}$. Since the order of $F(T^*)$ is $\alpha+1$, we may work directly on the skein $\text{Sk}(\alpha+1)$. Here, the point $F(a_i)$, which belongs to the thread $T_i$, is bound to either $x_{i}$ or $y_{i}$ in $T^i$. To see this, notice that, since there are no gaps in $T^*$ of length greater than $(2K)^{-1}/8$, by maximality of $[a_i,b_i]_{T^*}$ in $T_{\gamma^i}(x_i,y_i)$, we have that the distance from $F(a_i)$ to $\text{Sk}(\alpha+1)\setminus T^*$ is smaller than $1/8$. Hence, by construction of the successor ordinal skein $\text{Sk}(\alpha+1)$, the distance from $F(a_i)$ to one of the two extremes of the thread $T^i$ is also smaller than $1/8$, which implies that $F(a_i)$ is bound to one of these extremes. Similarly, $F(b_i)$ is bound to either $x_i$ or $y_i$ in $T^i$. Suppose without loss of generality that $F(a_i)$ is bound to $x_i$.

There are two possibilities: either $F(b_i)$ is bound to $x_i$ as well, or $F(b_i)$ is bound to the other extreme point $y_i$. If $F(b_i)$ is bound to $x_i$, then we can apply Proposition \ref{maximalwithentryeqexit} and obtain $F_i$ with the desired properties. 

\begin{figure}
    \includegraphics[width=\textwidth]{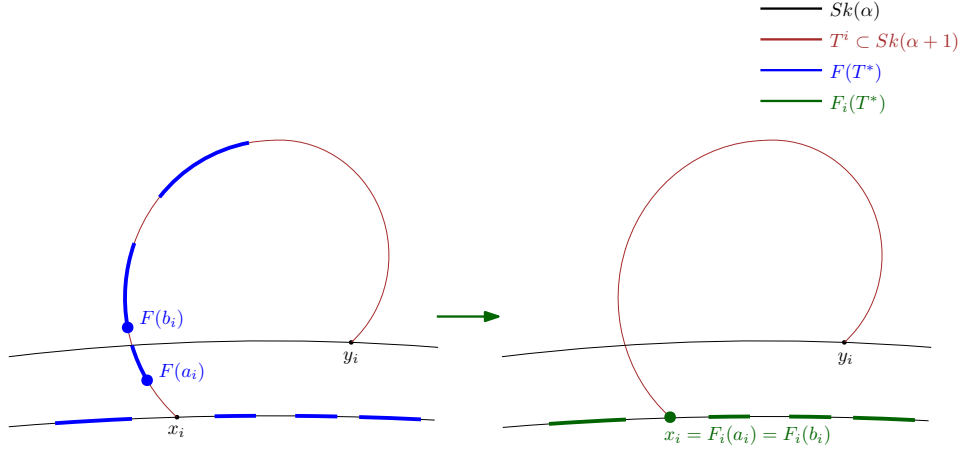}
    \caption{If both $F(a_i)$ and $F(b_i)$ are bound to the same anchor $x_i$, we may define $F_i$ by sending all points in $[a_i,b_i]_{T^*}$ to $x_i$ without increasing the Lipschitz constant.}
    \label{TransformingF_B_1}
\end{figure}

In Figure \ref{TransformingF_B_1} we observe this first possibility, and the resulting map $F_i$ according to Proposition \ref{maximalwithentryeqexit}.

Suppose now that $F(b_i)$ is bound to $y_i$ in $T^i$. We are going to show that there is a gap $C_i$ in $[a_i,b_i]_{T^*}$ with length greater than $d(x_i,y_i)/(K+\varepsilon)$. Indeed, suppose by contradiction there is no such gap.

We have that $F(a_i)$ belongs to the interval $[x_i,x_i+1/8)_{T^i}$, while $F(b_i)$ belongs to $(y_i-1/8,y_i]_{T^i}$. Recall the definition (prior to Fact 4) of the dense and countable subsets $D^i_1\subset [x_i,x_i+1/8)_{T^i}$ and $D^i_2\subset (y_i-1/8,y_i]_{T^i}$ in $T^i$, which were used to define the sequence of gaps of the thread $T^*$. Since $D^ i_1$ is dense in $[x_i,x_i+1/8)_{T^i}$, and $D^i_2$ is dense in $(y_i-1/8,y_i]_{T^i}$, considering the subset $[a_i,b_i]_{T^*}$ of $T^*$ as a thread, and restricting $F$ to this thread, we obtain by Proposition \ref{changingextremes} that there exist two points $a'_i,b'_i\in[a_i,b_i]_{T^*}$, and two points $(x'_,y'_i)\in D^i_1\times D^i_2$ with $x'_i<y'_i$, together with a $(K+\varepsilon)$-Lipschitz function $\overline{F}\colon [a'_i,b'_i]_{T^*}\rightarrow T_{\gamma^i}(x_i,y_i)$ such that $\overline{F}(a'_i)=x'_i$ and $\overline{F}(b'_i)=y'_i$. Notice that since the length of $T^i$ is $1$, and the points $x'_i$ and $y'_i$ belong to $[x_i,x_i+1/8)_{T^i}$ and $(y_i-1/8,y_i]_{T^i}$ respectively, the distance $d(x'_i,y'_i)$ is greater than $d(x_i,y_i)$. 

Finally, since we are assuming that there is no gap in $[a_i,b_i]_{T^*}$ with length greater than $d(x_i,y_i)/(K+\varepsilon)$, we can apply Proposition \ref{threadontosubinterval} and assume that $\overline{F}$ has its image contained in the thread $[x'_i,y'_i]_{T^i}$, which belongs to the family $\mathcal{T}$ we have used to define $\gamma^*$. Since the thread $[a'_i,b'_i]_{T^*}$ is a subinterval of $T^*$, its decreasing sequence of gaps $\{C^i_n\}_{n\in\mathbb{N}}$ also verifies that $\text{length}(C^i_n)<\gamma^*_n$ for all $n\in\mathbb{N}$. Hence, the existence of the function $\overline{F}$ whose Lipschitz constant does not exceed $K+\varepsilon< 2K$, implies by Fact \ref{choicesequence} that there is a gap $C^i_{n_0}$ in $[a'_i,b'_i]_{T^*}$ such that $\text{length}(C^i_{n_0})\geq d(x'_i,y'_i)/(K+\varepsilon)$.

The fact that the gap $C^i_{n_0}$ is also a gap of $[a_i,b_i]_{T^*}$ and that $d(x'_i,y'_i)\geq d(x_i,y_i)$ results in the desired contradiction.

Hence, there exist two points $c_i,d_i\in [a_i,b_i]_{T^*}$ with $c_i<d_i$ such that $(c_i,d_i)\cap (a_i,b_i)_{T^*}=\emptyset$ and $d(c_i,d_i)>d(x_i,y_i)/(K+\varepsilon)$. Define now $F_i\colon T^*\rightarrow \widehat{S}$ by 
$$
F_i(t)=
\begin{cases}
F(t)&\text{ if } t\in T^*\setminus [a_i,b_i]_{T^*},\\
x_i&\text{ if } t\in [a_i,c_i]_{T^*},\\
y_i&\text{ if } t\in [d_i,b_i]_{T^*}.
\end{cases}
$$
Using Proposition \ref{domainthreadisinterval}, maximality of $[a_i,b_i]_{T^*}$ for $F$ and $T_{\gamma^i}(x_i,y_i)$, and the fact that $F(a_i)$ and $F(b_i)$ are bound to $x_i$ and $y_i$ respectively in $T_{\gamma^i}(x_i,y_i)$, it is straightforward to check that $F_i$ verifies $\|F_i\|_\text{Lip}\leq K+\varepsilon$ (we use in fact the same argument as in the proof of Proposition \ref{maximalwithentryeqexit}).

\begin{figure}
    \includegraphics[width=\textwidth]{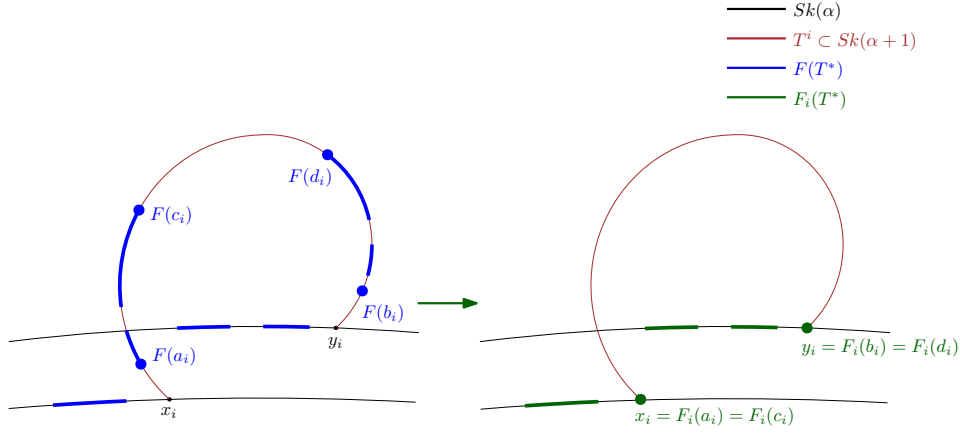}
    \caption{If $F(a_i)$ and $F(b_i)$ are bound to different anchors, then by choice of $T^*$ there must exist a gap $(c_i,d_i)$ in $[a_i,b_i]_{T^*}$ big enough to bridge the distance from $x_i$ to $y_i$ with a minimal increase of the Lipschitz constant.}
    \label{TransformingF_B_2}
\end{figure}

Figure \ref{TransformingF_B_2} intuitively summarizes the second possibility. Notice that in both Figures \ref{TransformingF_B_1} and \ref{TransformingF_B_2} the resulting map $F_i$ avoids the thread $T^i$, thus reducing the order of the image of the maximal interval $[a_i,b_i]_{T^*}$.
\end{proof}

To finish the proof of part (B) of the Claim, for each $t\in T^*$ such that $t\in [a_i,b_i]_{T^*}$ for some $i\in\mathbb{N}$, define $i(t)\in\mathbb{N}$ as the least of the natural numbers such that $t\in[a_{i(t)},b_{i(t)}]_{T^*}$. Now, define $\widehat{F}\colon T^*\rightarrow \widehat{S}$ by 
$$
\widehat{F}(t)=
\begin{cases}
F(t),&\text{ if } t\in T^*\setminus \bigg(\bigcup_{i\in\mathbb{N}}[a_i,b_i]_{T^*}\bigg),\\
F_{i(t)}(t),&\text{ if } t\in[a_i,b_i]_{T^*}\text{ for some }i\in\mathbb{N}.\\
\end{cases}
$$
To check that $\|\widehat{F}\|_\text{Lip}\leq K+\varepsilon$, we only need to consider $t,s\in T^*$ with $t<s$ and such that $t\in [a_{i(t)},b_{i(t)}]_{T^*}$ and $s\in[a_{i(s)},b_{i(s)}]_{T^*}$ with $i(t)\neq i(s)$. We have then the following inequalities:
\begin{align*}
    d\big(\widehat{F}(t),\widehat{F}(s)\big)&=d\big(F_{i(t)}(t),F_{i(s)}(s)\big)\\
    &\leq d\big(F_{i(t)}(t),y_{i(t)}\big)+d(y_{i(t)}+x_{i(s)})+d\big(x_{i(s)},F_{i(s)}(s)\big)\\
    &\leq d\big(F_{i(t)}(t),F_{i(t)}(b_{i(t)}\big)\big)+d\big(F(b_{i(t)}),F(a_{i(s)})\big)+d\big(F_{i(s)}(a_{i(s)}),F_{i(s)}(s)\big)\\
    &\leq \big(K+\varepsilon\big)\big((b_{i(t)}-t)+(a_{i(s)}-b_{i(t)})+(s-a_{i(s)})\big)\\
    &=\big(K+\varepsilon\big)(s-t).
\end{align*}
Since $d\big(\widehat{F}(p),\widehat{F}(q)\big) = d\big(P_{\beta}(p),P_\beta(q)\big)\leq d(p,q)=a_{T^*}$, we can  apply Proposition \ref{domainthreadisinterval} to obtain the desired Lipschitz constant for $\widehat{F}$ and finish the proof of (B).

\end{proof}
\begin{proof}[Proof of (C)]
The proof of the third case (C) resembles the proof of (B). The difference is that in this case at least one of $F(p)$ and $F(q)$ is in generation $\text{Sk}(\alpha+1)$, which is at the same time the order of $F(T^*)$. Intuitively, the idea of the proof of this last part is to first transform $F$ to lower the order of the image of $p$ and/or $q$. When we have done this, then we may simply apply the case (B) to the resulting map, thus obtaining a Lipschitz function whose image has a lower order than $F$.

Since $F(p)=P_{\alpha+1}(p)$, there exists $n\in\mathbb{N}$ such that $x_n=P_{\alpha}(p)$, $y_n=Q_\alpha(p)$, and $F(p)$ belongs to the thread $T^n=T_{\gamma^n}\big(P_\alpha(p),Q_\alpha(p)\big)\in\mathcal{T}_0$. We start by selecting $p'\in T^*$ such that $[p,p']_{T^*}$ is maximal for $F$ and the thread $T^n$. 

There are two possibilities: either $p'=q$, or the point $p'$ is different from $q$. If $p'=q$, since $P_\alpha(p)\neq P_\alpha(q)$, we have that the range of $F$ is contained in the single thread $T^n=T_{\gamma^n}\big(P_\alpha(p),P_\alpha(q)\big)$, and moreover $F(p)=P_{\alpha+1}(p)$ and $F(q)=P_{\alpha+1}(q)$ belong to this same thread. Since the distance from $p$ and $q$ to $\text{Sk}(\alpha)$ is less than $1/8$, we have that $F(p)$ belongs to $\big[P_{\alpha}(p),P_{\alpha}(p)+1/8\big)_{T^n}$ and $F(q)$ belongs to $\big[P_{\alpha}(q)-1/8,P_{\alpha}(q)\big)_{T^n}$. Hence, we can apply Propositions \ref{changingextremes} and \ref{threadontosubinterval} as we did in the proof of (B) to obtain two points $a',b'\in T^*$ with $a'<b'$ and two points $x',y'\in D^n_1\times D^n_2$ together with a $(K+\varepsilon)$-Lipschitz function $\overline{F}\colon [a',b']_{T^*}\rightarrow [x',y']_{T^n}$ with $\overline{F}(a')=x'$ and $\overline{F}(b')=y'$. Since the thread $[x',y']_{T^n}$ belongs to the family $\mathcal{T}$, by Theorem \ref{nolipschitzmap}, there exists a gap $C=(c,d)$ in $T^*$ such that $d(c,d)>d\big(P_\alpha(p),P_\alpha(q)\big)/(K+\varepsilon)$. Defining $\widehat{F}\colon T^*\rightarrow \widehat{S}$ as 
$$
\widehat{F}(t)=
\begin{cases}
P_{\alpha}(p),&\text{ if }t\in[p,c]_{T^*},\\
P_{\alpha}(q),&\text{ if }t\in[d,q]_{T^*}.
\end{cases}
$$
finishes the proof of (C) if $p'=q$, without need for further discussion. 

Hence, suppose now that $p'$ is not $q$. We are going to define a $(K+\varepsilon/2)$-Lipschitz function $\overline{F}_1\colon T^*\rightarrow \widehat{S}$ such that $\overline{F}_1(p)=P_{\alpha}(p)$ and $\overline{F}_1(t)=F(t)$ for all $t\in(p',q]_{T^*}$. 

In the space $\text{Sk}(\alpha+1)$, the point $F(p)=P_{\alpha+1}(p)$ is bound to $P_{\alpha}(p)$ in $T^n$ because the distance from $p$ to $\text{Sk}(\alpha)$ is less than $1/8$. In addition, the point $F(p')$ is also bound to one of the extremes $P_{\alpha}(p)$ or $Q_{\alpha}(q)$ in $T^n$. This is because there are no gaps in $T^*$ bigger than $(2K)^{-1}/8$ and $[p,p']_{T^*}$ is maximal for $F$ and $T^n$. We may consider again two possibilities: either $F(p')$ is bound to $P_{\alpha}(p)$ as well, or $F(p')$ is bound to $Q_{\alpha}(p)$. 

If $F(p')$ is bound to $P_{\alpha}(p)$, we can use Proposition \ref{maximalwithentryeqexit} to define a $K$-Lipschitz function $\overline{F}_1\colon T^*\rightarrow \widehat{S}$ with $\overline{F}_1(t)=P_{\alpha}(p)$ for all $t\in[p,p']_{T^*}$, and $\overline{F}_1(t)=F(t)$ for all $t\in (p',q]_{T^*}$; as desired. 

Suppose then that $F(p')$ is bound to $Q_{\alpha}(p)$. Then, since $F(p)\in\big[P_{\alpha}(p),P_{\alpha}(p)+1/8\big)_{T^n}$ and $F(p')\in \big[Q_{\alpha}(p)-1/8,Q_{\alpha}(q)\big)_{T^n}$, we can repeat the process we did in the proof of (B) and in the case when $p'=q$ to find a gap $C=(c,d)$ in $[p,p']_{T^*}$ such that $d(c,d)> d\big(P_\alpha(p),Q_\alpha(p)\big)/(K+\varepsilon/2)$. Again, we use this gap to define $\overline{F}_1\colon T^*\rightarrow \widehat{S}$ by 
$$
\overline{F}_1(t)=
\begin{cases}
P_{\alpha}(p),&\text{ if }t\in[p,c]_{T^*},\\
Q_{\alpha}(p),&\text{ if }t\in[d,p']_{T^*},\\
F(p),&\text{ if }t\in (p',q]_{T^*}.
\end{cases}
$$
The Lipschitz constant of $\overline{F}_1$ is less than or equal to $(K+\varepsilon/2)$ as desired. 

We may repeat the same argument to find a point $q'\in T^*$ with $p'<q'$, together with a second $(K+\varepsilon/2)$-Lipschitz function $\overline{F}_2\colon T^*\rightarrow\widehat{S}$ such that $\overline{F}_2(q)=P_{\alpha}(q)$ and $\overline{F}_2(t)=F(t)$ for all $t\in [p,q')_{T^*}$. We combine $\overline{F}_1$ and $\overline{F}_2$ to form yet another Lipschitz function $\overline{F}\colon T^*\rightarrow \widehat{S}$ in the following way:
$$ 
\overline{F}(t)=
\begin{cases}
\overline{F}_1(t),&\text{ if }t\in[p,p']_{T^*},\\
F(t),&\text{ if }t\in(p',q')_{T^*},\\
\overline{F}_2(t),&\text{ if }t\in (q',q]_{T^*}.
\end{cases}
$$
It is again straightforward to prove that the Lipschitz constant of $\overline{F}$ is less than or equal to $K+\varepsilon/2$. It is possible that the order of $\overline{F}(T^*)$ is already the desired ordinal $\alpha<\alpha+1$, in which case the proof is finished. However, it might be that there are points in $\overline{F}(T^*)$ in the generation $G_{\alpha+1}$. If this is the case, notice that the function $\overline{F}$ verifies the hypothesis of the claim and the conditions of (B). Hence, we may use the already proven case (B) with $\varepsilon/2>0$ to find a $(K+\varepsilon)$-Lipschitz function $\widehat{F}\colon T^*\rightarrow\widehat{S}$ such that $\widehat{F}(p)=P_{\alpha}(p)$, $\widehat{F}(q)=P_{\alpha}(q)$, and the order of $\widehat{F}(T^*)$ is $\alpha$. The proof is now finished.
\end{proof}
We have proven the three parts of the claim.
\end{proof}
Having proven the claim, the theorem holds by the discussion after the statement of the claim.

\end{proof}
\printbibliography
\end{document}